\RequirePackage{ifthen}
\RequirePackage{ifpdf}

\ifthenelse{\boolean{pdf}}{
	
} { % else

}

\documentclass[11pt,letter]{article}

\newboolean{isCommented}
\usepackage{amsmath}
\usepackage{amsthm}
\usepackage{geometry}
\usepackage{graphicx}
\usepackage{subfig}
\usepackage{color}
\usepackage{tikz}
\usepackage{enumerate}
\usepackage{amssymb}
\usepackage[utf8x]{inputenc}
\usepackage{float}
\usepackage{hyperref}
\newdimen\progindentamount\progindentamount2em
\newdimen\progindentamount\progindentamount8mm
\def\i#1{\strut\egroup\egroup\hbox to\hsize\bgroup
         \vrule width0pt\hskip#1\progindentamount
         \vtop\bgroup\advance\hsize-#1\progindentamount\strut\ignorespaces}
\newenvironment{algorithm}{\par\begingroup\def\\{hfill\break}
                           \setbox0\hbox\bgroup\bgroup
                         }{\strut\egroup\egroup\endgroup
                         }
\floatstyle{ruled}
\newfloat{figure-algo}{h}{lof}[section]
\floatname{figure-algo}{Algorithm}

\newenvironment{algo-io}[4]{%
  \begin{figure-algo}
  \caption{\label{alg:#1} {#2}}
  \begin{list}{}{\leftmargin 16mm\labelwidth15mm\labelsep0.1cm\topsep0pt\parsep0pt\itemsep 0pt}
    \item[Input:\hfill] {#3}
    \item[Output:\hfill] {#4}
      \hspace{0.2cm}
  \end{list}
  \begin{algorithm}
}
{%
  \end{algorithm}
  \end{figure-algo}
}

\ifthenelse{\boolean{pdf}}{
	
} { % else
	
}

\newtheorem{theorem}{Theorem}
\newtheorem{lemma}[theorem]{Lemma}

\newtheorem{definition}{Definition}

\newcommand{\multilinecomment}[1]{}

\let\Exp=\EE

\newcommand{\Bin}{\mathrm{Bin}}

\def\clap#1{\hbox to 0pt{\hss#1\hss}}
\def\mathclap{\mathpalette\mathclapinternal}
\def\mathclapinternal#1#2{\clap{$\mathsurround=0pt#1{#2}$}}

\long\def\symbolfootnote[#1]#2{\begingroup
\def\thefootnote{\fnsymbol{footnote}}\footnote[#1]{#2}\endgroup}

\ifthenelse{\boolean{isCommented}} {
	\geometry{%showframe,
	hmargin={25mm, 50mm},
	marginparwidth=40mm,
	vmargin={25mm, 25mm},
	headsep=10mm,
	headheight=5mm,
	footskip=10mm
	}
} { % else
	\geometry{%showframe,
	hmargin={1in, 1in},
	vmargin={1in, 1in},
	headsep=10mm,
	headheight=5mm,
	footskip=10mm
	}
}

\title{Random Hyperbolic Graphs: Degree Sequence and Clustering\\\small{Full Version}}

\author{Luca Gugelmann \\ \small{Institute of Theoretical Computer Science} \\ \small{ETH Zurich, 8092 Zurich, Switzerland} \\\small{\texttt{lgugelmann@inf.ethz.ch}}
\and Konstantinos Panagiotou \\ \small{Department of Mathematics} \\ \small{University of Munich, 80333 Munich, Germany} \\\small{\texttt{kpanagio@mpi-inf.mpg.de}} \and Ueli Peter \\ \small{Institute of Theoretical Computer Science} \\ \small{ETH Zurich, 8092 Zurich, Switzerland} \\\small{\texttt{upeter@inf.ethz.ch}}}
\date{\today}

\numberwithin{equation}{section}
\numberwithin{theorem}{section}
 
\begin{document}

\maketitle 

\thispagestyle{empty}

\begin{abstract}
In the last decades, the study of models for large real-world networks has been a very popular and active area of research. A reasonable model should not only replicate all the structural properties that are observed in real world networks (for example, heavy tailed degree distributions, high clustering and small diameter), but it should also be amenable to mathematical analysis. There are plenty of models that succeed in the first task but are hard to analyze rigorously. On the other hand, a multitude of proposed models, like classical random graphs, can be studied mathematically, but fail in creating certain aspects that are observed in real-world networks.

Recently, Papadopoulos, Krioukov, Bogu\~n\'{a} and Vahdat [INFOCOM'10] introduced a random geometric graph model that is based on hyperbolic geometry. The authors argued empirically and by some preliminary mathematical analysis that the resulting graphs have many of the desired properties. Moreover, by computing explicitly a maximum likelihood fit of the Internet graph, they demonstrated impressively that this model is adequate for reproducing the structure of real graphs with high accuracy.

In this work we initiate the rigorous study of random hyperbolic  graphs. We compute exact asymptotic expressions for the expected number of vertices of degree $k$ for all $k$ up to the maximum degree and provide small probabilities for large deviations. We also prove a constant lower bound for the clustering coefficient. In particular, our findings confirm rigorously that the degree sequence follows a power-law distribution with controllable exponent and that the clustering is nonvanishing. 
\end{abstract}

\newpage

\section{Introduction}
Modeling the topology of large networks is a fundamental problem that has attracted considerable attention in the last decades. Networks provide an abstract way of describing relationships and interactions between elements of complex and heterogeneous systems. Examples include technological networks like the World Wide Web or the Internet, biological networks like the human brain, and social networks which describe various kinds of interactions between individuals.

An accurate mathematical model can have enormous impact on several research areas. From the viewpoint of computer science, an obvious benefit is that it could enable us to design more efficient algorithms that exploit the underlying structures. Moreover, the process of modeling may suggest and reveal novel types of qualitative network features, which become patterns to look for in datasets. Finally, an appropriate model allows us to generate artificial instances, which resemble realistic instances to a high degree, for simulation purposes. Unfortunately, from today's point of view, a significant proportion of the current literature is devoted only to experimental studies of properties of real-world networks, and there has been only little rigorous mathematical work.

There are (at least) two requirements for a reasonable model for real-world networks. First, it must be able, when setting the parameters appropriately, to replicate the salient features of the real-world graphs under consideration. Moreover, a second desired property is that the model should be mathematically tractable and simple enough to be of use in large scale simulations. There are plenty of models that satisfy the first criterion, but are hard to analyze from a mathematical viewpoint. On the other hand, there exists a plethora of analytically tractable models, which unfortunately do not yet replicate satisfactory enough the properties that are observed in large networks.

In this work we initiate the rigorous study of a class of models for large networks, the so-called~\emph{random hyperbolic graphs}. Such graphs were shown empirically to have startling similarities with several real-world networks, and in particular with the Internet graph (i.e., the network formed by the routers and their physical connections). Before we describe the model and our results, let us proceed with considering some properties of large networks in a little more detail.

\vspace{-10pt}\paragraph{Properties of large networks.}
Since the 60's, the study of networks of various kinds has grown into a significant research area. One of the initiators in this field, the sociologist Stanley Milgram, investigated the network that is obtained from the relationships among people~\cite{milgram1967small, travers1969experimental}. In his work he discovered what is nowadays known as the \emph{small-world phenomenon}, which postulated that the distance between two random people is on average between five and six. Outside the context of social networks this has become synonymous to a graph with a comparatively low diameter/average path length, and nowadays many networks are known to possess this property~\cite{albert1999internet, amaral2000classes,watts1998}.

Another property that is found in many networks addresses the degree distribution. In a celebrated paper, Faloutsos et.\ al.\ \cite{faloutsos1999power} observed that the Internet exhibits a so-called scale-free nature: the degree sequence follows approximately a \emph{power-law} distribution, which means that the number of vertices of degree $k$ is proportional to some inverse power of $k$, for all sufficiently large $k$. This sets such a network dramatically apart from e.g.\ a typical Erd\H{o}s-R\'{e}nyi random graph and stirred significant interest in exploring the causes of this phenomenon. From today's viewpoint, it is well-known that many graphs have a \emph{heavy-tailed} degree distribution, which may be close to a power-law or a log-normal or a combination of these distributions (see \cite{mitzenmacher2004brief} and references therein).

A third distinctive feature of large real-world graphs is the appearance \emph{clustering} \cite{newman2003social, serrano2006clustering,watts1998}. The network average of the probability that two neighbors of a random vertex are also directly connected is called the clustering coefficient. Measured clustering coefficients for social networks are typically tens of percent, and similar values have been measured for many other networks as well, including technological and biological ones.

\vspace{-10pt}\paragraph{Models of large networks.} Perhaps the first step towards a random graph model for real-world networks was made by Watts and Strogatz \cite{watts1998} in 1998, who addressed the small-world phenomenon and clustering and gave reasons for its emergence. However, the degree distribution of the generated graphs follows a Poisson distribution, and thus is not heavy tailed. Barab\'asi and Albert proposed \cite{barabasi1999emergence} that the cause for power-law degree distributions is preferential attachment: the networks evolve continuously by the addition of new vertices, and each new vertex chooses its neighbors with a probability that is proportional to their current degree. This model was shown by Bollob\'{a}s et al.~\cite{bollobas2001degree} to produce power-law degree distributions, but on the other hand it generates graphs that typically have a vanishing clustering coefficient \cite{bollobas2002mathematical}. Nevertheless, the Barab\'asi-Albert model was the beginning of a vast series of proposed models that suggested mechanisms according to which a network can evolve (see e.g. \cite{aiello2001random, borgs2007first, buckley2004popularity, chierichetti2009models, cooper2003general, lattanzi2009affiliation,leskovec2005graphs} for a non-exhaustive but representative list).% \marginpar{\tiny add affiliation networks here}

\vspace{-10pt}\paragraph{Hyperbolic random graphs.} 
An alternative and fruitful approach towards understanding the structure and the dynamics of real-world networks is to attempt to describe the similarities or dissimilarities between vertices in a well-defined and formal sense. One possibility in this direction is based on the idea of assigning \emph{virtual coordinates} to the vertices, i.e., the network is embedded in some metric space such that the mutual distances abstract the resemblance among the vertices.

One natural choice for the underlying metric space is the Euclidean space. In this context, Ng et al.~\cite{ng2002} proposed to embed the Internet graph into such spaces. Their original aim was to predict distances in the network by simply comparing coordinates. The authors obtained a reasonable mapping in 5 and 7 dimensions, but not without distortion and errors. Shavitt and Tankel \cite{shavitt2004curvature} later observed that this embedding becomes dramatically better when replacing the Euclidean geometry with a negatively curved hyperbolic space.

The above considerations lead us immediately to the model of \emph{Random Geometric Graphs}. Such a graph is generated by placing independently and uniformly at random $n$ vertices in, say, $[0,1]^2$, and creating edges whenever the (Euclidean) distance of two vertices is at most some $r = r(n)$. These graphs have been studied intensively by many authors because of connections to percolation, statistical physics, hypothesis testing, and cluster analysis~\cite{penrose2003random}. Unfortunately, these results provide strong evidence that Euclidean geometry is not the adequate choice if one wants to describe large real-world networks, as the qualitative characteristics of the resulting random networks (like the average path length or the degree sequence) are very far from the ones observed in practice. In other words, the underlying geometry capturing the main structural characteristics of real-world networks is not Euclidean, and the important question is whether there exists an appropriate choice of a geometry giving rise to the observed features.

A preliminary answer to this question was given by Papadopoulos, Krioukov, Bogu\~{n}\'{a} and Vahdat~\cite{papadopoulos2010greedyforwarding}.
The authors demonstrated impressively that complex scale-free network topologies with high clustering
coefficients emerge naturally from hyperbolic metric spaces.  Their model, which we will denote by \emph{random hyperbolic graph}, consists in its simplest variant of the uniform distribution of \(n\) vertices within a disk of radius \(R = R(n)\) in the hyperbolic plane, where two vertices are connected if their hyperbolic distance is at most \(R\). The authors show via simulations and some preliminary theoretical analysis that the generated graphs exhibit a power-law degree distribution, whose exponent can be tweaked via model parameters. Further, the authors indicate that with a slightly more complex model they can also control the clustering of the generated graphs to bring it in line with real-world networks.

To make their case, Bogu\~{n}\'{a}, Papadopoulos and Krioukov computed in~\cite{boguna2010sustaining} an embedding of the Internet graph into the hyperbolic plane by finding the maximum likelihood match to the model that was described above, and demonstrated impressively that this embedding has many desirable properties. For example, the authors examined the performance of greedy routing using the hyperbolic coordinates, i.e.\ the scheme in which each node forwards an incoming message to a neighbor that is closest to the destination, see also the works of Kleinberg~\cite{kleinberg2007geographic} and Papadimitriou et al.~\cite{papadimitriou2005conjecture}. In their embedding, this simple greedy forwarding strategy exhibits a remarkably strong performance and connects 97\% of all vertex pairs. The average stretch factor between chosen and optimal path is around 1.1, suggesting that greedy paths are very close to optimal. They also showed that this performance remains strong even if a fraction of the nodes is allowed to fail. 

\vspace{-10pt}\paragraph{Our contribution.}
Regarding the experiments just described, it seems at least fair to say that random hyperbolic graphs provide an attractive model that has a high potential of being adequate for describing the characteristics of many real-world networks. 
Moreover, a simple formulation and a strong affinity to random geometric graphs indicate that this model might be mathematically tractable. In this work we show that this is indeed the case and initiate thereby the rigorous study of hyperbolic random graphs. First, we prove a constant lower bound on the clustering coefficient of hyperbolic random graphs which confirms the claimed high clustering. We then show that the expected degree distribution indeed follows a power-law \emph{across all scales}, i.e., even up to the \emph{maximum degree}. Note that in the seminal papers \cite{papadopoulos2010greedyforwarding} and \cite{krioukov2010hyperbolic} the degree distribution was also considered, however only for constant degrees and without any error guarantees. In addition, we prove small probabilities for large deviations, i.e.\ we show that sampling from this distribution returns with high probability a graph with the desired properties, which is crucial for validating experimental results. We also compute tight bounds for the average and maximum degree that hold with high probability.

There are many models for which either a power-law degree sequence \cite{aiello2001random, bollobas2001degree, buckley2004popularity} or a large clustering coefficient \cite{watts1998} has been proven. But this is the first model which provably satisfies both properties. 
Note also that while there are some models (see for example~\cite{bollobas2001degree} and~\cite{buckley2004popularity}) for which a power-law degree distribution up to polynomially large degrees can be showed, to the best of our knowledge this is the first rigorous proof that the degree distribution of a random graph model is scale-free \emph{up to the maximum degree}. Further, our results reveal some fundamental combinatorial properties of the model, thus setting the groundwork for further theoretical investigations. We strongly believe that these facts together with the nice combinatorial structure of the model make it attractive for the theoretical computer science and random graph community.

\section{Model \& Results}
Let us begin this section with a few facts about the geometry of hyperbolic planes. We will restrict ourselves to the most basic notions, and refer the reader to e.g.~\cite{book:a05} and many references therein for an extensive introduction.

First of all, there are many equivalent representations of the hyperbolic plane, each one highlighting different aspects of the underlying geometry. We will consider here the so-called~\emph{native} representation, which was described by~Papadopoulos et.\ al.\ in~\cite{papadopoulos2010greedyforwarding}, as it is most convenient for defining the model of random hyperbolic graphs.

One basic feature of the hyperbolic plane is that it is isotropic, meaning that the geometry is the same regardless of direction. In other words, we can distinguish an arbitrary point, which we  call the~\emph{center} or the \emph{origin}. In the native representation of the hyperbolic plane we will use polar coordinates~$(r, \theta)$ to specify the position of any vertex~$v$, where the radial coordinate~$r$ equals the hyperbolic distance of~$v$ from the origin. Given this notation, the distance~$d$ of two vertices with coordinates~$(r, \theta)$ and~$(r', \theta')$ can be computed by solving the equation
\begin{equation}
\label{eq:hyp_dist}
\cosh(d)=\cosh(r) \cosh(r')-\sinh(r)\sinh(r')\cos(\theta -\theta') ,
\end{equation}
where~$\cosh(x) = (e^x + e^{-x})/2$ and~$\sinh(x) = (e^x - e^{-x})/2$. For our purposes we will denote from now on by~$d(r,r',\theta- \theta')$ the solution of \eqref{eq:hyp_dist} for~$d$.

The crucial difference between the Euclidean and the hyperbolic plane is that the latter contains in a well-defined sense more ``space''. More specifically, a circle with radius~$r$ has in the Euclidean plane a length of~$2\pi r$, while its length in the hyperbolic plane is~$2\pi\sinh(r) = \Omega(e^r)$. In other words, a circle in the hyperbolic plane has a length that is exponential in its radius as opposed to linear.

Based on the above facts, the authors of~\cite{papadopoulos2010greedyforwarding} defined a model of random geometric graphs that in its simplest version consists of the uniform distribution of \(n\) points into a hyperbolic disk of radius \(R = R(n)\) around the origin.  Two points in this disk are connected by an edge only if they are at hyperbolic distance at most \(R\) from each other, as defined in~\eqref{eq:hyp_dist}. More precisely, note that the total area of a circle of radius $r$ equals
\begin{equation*}
  2\pi \int_{0}^r \sinh(t) dt = 2\pi (\cosh(r) - 1).
\end{equation*}
To choose the \(n\) points uniformly at random in the hyperbolic disk of radius \(R\) it suffices to choose for each polar coordinates \((r, \theta)\) such that \(\theta\) is chosen uniformly at random in the interval, say, \((-\pi,\pi]\) and its radial coordinate \(r\) is drawn according to the distribution with density function $\sinh(r)/(\cosh(R) - 1)$, where \(0\leq r \leq R\). To add flexibility to the model, the authors of~\cite{papadopoulos2010greedyforwarding} use a slightly different density function for the radial coordinate: \(\alpha \sinh(\alpha r)/(\cosh(\alpha R) - 1)\), where \(\alpha > 1/2\). For \(\alpha < 1\) this favors points closer to the center, while for \(\alpha > 1\) points with radius closer to \(R\) are favored. For \(\alpha = 1\) this corresponds to the uniform distribution.

Let us now proceed to a formal definition of the model. With all the above notation at hand, the random hyperbolic graph $G_{\alpha, C}(n)$ with $n$ vertices and parameters $\alpha$ and $C$ is defined as follows.
\begin{definition}[Random Hyperbolic Graph $G_{\alpha, C}(n)$]
\label{model:hyperbolic}
Let $\alpha > 1/2$, $C \in \mathbb{R}$, $n\in \mathbb{N}$, and set $R = 2\log n + C$. The random hyperbolic graph $G_{\alpha, C}(n)$ has the following properties.
\begin{itemize}
\item The vertex set $V$ of $G_{\alpha, C}(n)$ is $V = \{1, \dots, n\}$.
\item Every $v\in V$ is equipped with random polar coordinates $(r_v, \theta_v)$, where $r_v\in [0, R]$ has density $p(r):=\alpha\frac{\sinh(\alpha r)}{\cosh (\alpha R) -1}$ and $\theta_v$ is drawn uniformly from $[-\pi, \pi]$.
\item The edge set of $G_{\alpha, C}(n)$ is given by $\bigl\{\{u,v\}\subset \binom{V}2:  d(r_u,r_v, \theta_u - \theta_v )\le R\bigr\}$.
\end{itemize}
\end{definition}
The restrictions in the model parameters, especially the condition $\alpha > 1/2$ and the definition of $R$ will become clear in the sequel. Informally speaking, the choice of $R$ guarantees that the resulting graph has a bounded average degree (depending on $\alpha$ and $C$ only). If $\alpha \le 1/2$, then the degree sequence is so heavy tailed that this is impossible.

Let us mention at this point that in~\cite{papadopoulos2010greedyforwarding} an even more general model was also proposed. There, each pair of vertices is connected with a probability that may depend on the hyperbolic distance of those vertices. In particular, this probability is large if the vertices have distance $\le R$, and becomes quickly smaller when the distance is larger than $R$. We will not treat this model here.

Let us next describe the results that we show for $G_{\alpha, C}(n)$. First of all, we study the clustering of hyperbolic random graphs. The local clustering coefficient of a vertex $v$ is defined by 
\begin{equation}
\label{eq:def_local_clustering}
  \bar{c}_v =
    \begin{cases}
      0 & \text{ if \(\deg(v) < 2\),}\\
      \frac{\left|\{\{u_1, u_2\}\in E\mid u_1, u_2\in \Gamma(v)\} \right|}{\binom{\deg(v)}{2}}  & \text{else}
    \end{cases}
\end{equation}
where $\Gamma(v):=\{u\mid\{u,v\}\in E\}$. The global clustering coefficient of a graph $G=(V,E)$ is the average over all local clustering coefficients
\begin{equation}
\label{eq:def_global_clustering}
\overline{c}(G)=\frac{1}{n}\sum_{v\in V}\bar{c}_v.
\end{equation}
Our first theorem gives a constant lower bound on the global clustering coefficient which holds with high probability. 
\begin{theorem}
\label{thm:clustering}
Let $\alpha > 1/2$, $C\in\mathbb{R}$ and $\bar{c}=\bar{c}(G_{\alpha, C}(n))$. Then $\mathbb{E}[\overline{c}]=\Theta(1)$ and with high probability $\overline{c}=(1+o(1))\mathbb{E}[\overline{c}]$.
\end{theorem}
Then we study the degree sequence, and provide sharp bounds for the number of vertices of degree $k$.
\begin{theorem}
\label{thm:degree_seq}
Let $\alpha > 1/2$ and $C\in\mathbb{R}$. Set $\delta = \min \left\{{\frac{(2\alpha-1)}{4(2\alpha+1)\alpha}}, {\frac{2(2\alpha-1)}{5(2\alpha+1)}} \right\}$. Then, with high probability, for all $0 \le k \le n^{\delta'}$, where $\delta'<\delta$, the fraction of vertices of degree exactly $k$ in $G_{\alpha, C}(n)$ is 
\begin{equation}
\label{eq:theorem1_part1}
\bigl(1+o(1)\bigr)\frac{2\alpha e^{-\alpha C}}{k!} \Bigl(\frac{2\alpha}{\pi(\alpha-1/2)}\Bigr)^{2\alpha}\left( \Gamma(k-2\alpha)-\int_{0}^{\xi} t^{k-2\alpha-1}e^{-t}dt  \right),
\end{equation}
where 
$\Gamma(x) = \int_{0}^\infty t^{x-1}e^{-t}dt$ denotes the Gamma function and $\xi=\frac{2\alpha}{\pi(\alpha - 1/2)}e^{-C/2}$.
If $n^{\delta}\le k\le \frac{n^{1/2\alpha}}{\log n}$ then with high probability the fraction of vertices of degree at least $k$ in $G_{\alpha, C}(n)$ is
\begin{equation}
\label{eq:theorem1_part2}
(1+o(1))\left(\frac{2\alpha}{\pi(\alpha-1/2)} \right)^{2\alpha}e^{-\alpha C}k^{-2\alpha} .
\end{equation}
\end{theorem}
Note that this result demonstrates that the degree sequence of $G_{\alpha, C}(n)$ is a power-law with exponent $2\alpha + 1 > 2$. To see this, note that for sufficiently large $k$ we have that $\Gamma(k - 2\alpha)/k! = \Theta(k^{-2\alpha -1})$, i.e., \eqref{eq:theorem1_part1} and \eqref{eq:theorem1_part2} imply that the number of vertices of degree $k$ in $G_{\alpha, C}(n)$ is $(1+o(1)) c_{\alpha, C} k^{-2\alpha - 1}n$, for an appropriate $c_{\alpha, C} > 0$.

%in $G_{\alpha, C}(n)$, the density function of the degree sequence for small degrees \eqref{eq:theorem1_part1} and the derivative of the cumulative distribution for large degrees \eqref{eq:theorem1_part2} are both $(1+o(1)) c_{\alpha, C} k^{-2\alpha - 1}$, for an appropriate $c_{\alpha, C} > 0$.

Our next result gives bounds for the average degree of $G_{\alpha, C}(n)$.
\begin{theorem}
\label{thm:average_degree}
Let $\alpha > 1/2$ and $C\in\mathbb{R}$. Then the average degree of $G_{\alpha, C}(n)$ is $(1+o(1))\frac{2\alpha^2e^{-C/2}}{\pi (\alpha -1/2)^2}.$
\end{theorem}
Note that Theorem~\ref{thm:degree_seq} and Theorem~\ref{thm:average_degree} confirm the results in~\cite{papadopoulos2010greedyforwarding}.
Finally, we give sharp bounds for the maximum degree in $G_{\alpha, C}(n)$.
\begin{theorem}
\label{thm:max_degree}
Let $\alpha > 1/2$ and $C\in\mathbb{R}$. Then the maximum vertex degree of $G_{\alpha, C}(n)$ is with high probability  $n^{\frac{1}{2\alpha}+o(1)}$.
\end{theorem}

In this work we focus on the degree distribution and the clustering of hyperbolic random graphs. We believe that the the diameter of the giant component and the performance of greedy routing using the hyperbolic coordinates are interesting questions for future work in this area.

\section{Properties of the Model}
\label{sec:properties}
Recall that according to Definition~\ref{model:hyperbolic} the mass of a point $p=(r, \theta)$ is $f(r)=\frac{\alpha \sinh(\alpha r)}{2\pi(\cosh(\alpha R)-1)}$, and does only depend on the radial coordinate of $p$. Accordingly, we define the probability measure $\mu(S)$ of a point set $S$ as
\begin{equation}
\label{eq:pointset_mu}
\mu(S)=\int_Sf(y)dy .
\end{equation}
A vertex located at $(\theta, r)$ is connected to all vertices with coordinates $(\theta', r')$ such that $d(r, r', \theta-\theta')\le R$. Let us define the \emph{ball} of radius $x$ around a point $(r, \theta)$ as 
\begin{equation}
\label{eq:ball}
B_{r, \theta}(x)=\left\{ (r', \theta') ~\big|~ d(r,r',\theta-\theta')\le x \right\}. 
\end{equation}
Since in the definition of our model two points are connected if and only if they are at distance at most $R$, we will typically consider the intersection $B_{r_1, \theta_1}(R)\cap B_{0,0}(R)$ which corresponds to the point set in which all vertices are connected to a fixed vertex at $(r_1, \theta_1)$.
By \eqref{eq:pointset_mu} we can determine the probability measure of such a set by integrating $f(y)$ over all points in the set. In our specific case we achieve this by integrating first over all $y\in [0,R]$ and then over all $\theta$ such that $d(r_1, y, \theta_1-\theta)\le R$. As $f(y)$ does not depend on $\theta$ we are only interested in the range of $\theta$ for which this inequality is satisfied. One extremal of $(\theta_1-\theta)$ for which it is satisfied is clearly
\begin{equation}
\label{eq:def_theta_y}
\theta_{r_1}(y)= \arg\max_{0\le \phi\le \pi}\left\{d(r_1, y, \phi)\le R \right\}  =\arccos\left(\frac{\cosh(r_1)\cosh(y)-\cosh(R)}{\sinh(r_1)\sinh(y)} \right).
\end{equation}
Because of symmetry of the cosine the other extremal is $-\theta_{r_1}(y)$ and we therefore have to integrate from $-\theta_{r_1}(y)$ to $\theta_{r_1}(y)$ as all those angles $\theta$ satisfy $d(r_1,y,\theta)\le R$. Therefore we arrive at the following expression:
\begin{equation}
\label{eq:ball_integral}
\mu(B_{r, \theta}(R)\cap B_{0,0}(R))=\int_0^R\int_{-\theta_r(y)}^{\theta_r(y)} f(y)d\theta dy=2\int_0^R\int_{0}^{\theta_r(y)} f(y)d\theta dy.
\end{equation}
Note that $\mu(B_{r, \theta}(x)\cap B_{0,0}(R))$ does not depend on $\theta$ and therefore we shorten it to $\mu(B_{r}(x)\cap B_{0}(R))$. Before we commence with the more technical part of this section, we quickly refresh the following basic estimates of $\cosh(x)$ and $\sinh(x)$. For all $x\geq 0$
\begin{equation}
\label{eq:approx_cosh}
\frac{e^{x}}{2}\le \cosh(x)\le e^x \quad \text{and} \quad \frac{e^x}{3}\stackrel{(x\geq 1/2\ln 3)}{\le} \sinh(x) \le \frac{e^x}{2}  .
\end{equation}
We first prove a technical lemma that gives almost tight bounds on $\theta_{r}(y)$.

\begin{lemma}
\label{lem:theta}
Let $0\le r \le R$ and $y\geq R-r$. Then
$$\theta_r(y)= 2e^{\frac{R-r-y}{2}}\left(1+\Theta\left(e^{R-r-y}\right) \right).$$ 
\end{lemma}

\begin{proof}
 By using~\eqref{eq:def_theta_y} and the trigonometric identity
  \begin{equation}
    \label{eq:trig_id_cosh_add}
    \cosh(x\pm y)=\cosh(x)\cosh(y)\pm \sinh(x)\sinh(y) .
  \end{equation}
  we infer that
  \begin{equation*}
    \begin{split}
      \cos\bigl(\theta_r(y)\bigr) &= \frac{\sinh(r)\sinh(y) +\cosh(r-y)
        -\cosh(R)}{\sinh(r)\sinh(y)}\\
      &=1 + 2\frac{e^{r-y}+e^{-r+y}}{(e^{r}-e^{-r})(e^{y}-e^{-y})} -
      2\frac{e^{R}+e^{-R}}{(e^{r}-e^{-r})(e^{y}-e^{-y})}\\
      &=1 + 2\frac{e^{-2r}+e^{-2y}}{(1-e^{-2r})(1-e^{-2y})} -
      2\frac{e^{R-r-y}+e^{-R-r-y}}{(1-e^{-2r})(1-e^{-2y})}.
    \end{split}
  \end{equation*}
  Observe that both \(r\) and \(y\) are non-negative. By applying the identity \(1/(1-x) = 1 + \Theta(x)\), which is valid for all $0< x <1$, we obtain
  \begin{equation*}
    \begin{split}
      \cos\bigl(\theta_r(y)\bigr) &\stackrel{\phantom{y \ge R-r}}{=} 1 + 2 \left(e^{-2y} + e^{-2r} - e^{R-r-y} - e^{-R-r-y}\right)\left(1+\Theta(e^{-2r})\right)\left(1+\Theta(e^{-2y})\right)\\
      &\stackrel{y \ge R-r}{=} 1 - 2e^{R-r-y} + \Theta(e^{-2y} + e^{-2r}).
    \end{split}
  \end{equation*}
In our next estimates we will get rid of the cosine in the above expression.
By $\cos(\theta)\geq 1-\frac{\theta^2}{2}$ we derive
$$\theta_r(y)^2\geq 4 e^{R-r-y}  -\Theta\left(e^{-2r}+e^{-2y} \right).$$
Note that whenever $|x| \le 1$
\begin{equation}
\label{eq:sqrt_approx}
\sqrt{1+x}=1+\frac{x}{2}+\Theta(x^2).
\end{equation}
Applied to the previous equation, this gives a lower bound of
\begin{equation}
\label{eq:theta_lb}
\theta_r(y)\geq 2 e^{\frac{R-r-y}{2}}\left(1-\Theta\left(\frac{e^{-2r}+e^{-2y}}{e^{R-r-y}}\right)\right).
\end{equation}
Next we will derive an almost matching upper bound for $\theta_r(y)$. First we exploit that $\cos(\theta)\le 1-\frac{\theta^2}{2}+\frac{\theta^4}{4!}$ and thereby
$$ \frac{\theta_r^2(y)}{2}-\frac{\theta_r^4(y)}{4!} \le  2e^{R-r-y} - \Theta\left(e^{-2r}+e^{-2y}\right) .$$ 
This quadratic equation can be solved exactly by using basic tools. We omit the detailed calculations, and show just the final outcome. We obtain that
\begin{equation*}
\theta_r(y)^2\le 6-6\sqrt{1-\frac{4}{3} e^{R-r-y}\left(1-\Theta\left(\frac{e^{-2r}+e^{-2y}}{e^{R-r-y}}\right)\right)}.
\end{equation*}
Note that if $\frac{4}{3}e^{R-r-y} \left(1-\Theta\left(\frac{e^{-2r}+e^{-2y}}{e^{R-r-y}}\right)\right)<1$ then we can apply \eqref{eq:sqrt_approx}. If on the other hand we have $\frac{4}{3}e^{R-r-y} \left(1-\Theta\left(\frac{e^{-2r}+e^{-2y}}{e^{R-r-y}}\right)\right)\geq1$ then $\Theta\left(\left( e^{R-r-y}-(e^{-2r}+e^{-2y})\right)^2 \right)=\Theta(1)$. With this we obtain
\begin{align*}
  \theta_r(y)^2& \le 4 e^{R-r-y}\left(1-\Theta\left(\frac{e^{-2r}+e^{-2y}}{e^{R-r-y}}\right)\right)+\Theta\left(\left( e^{R-r-y}-(e^{-2r}+e^{-2y})\right)^2 \right)\\
&=\, 4  e^{R-r-y} -\Theta\left(e^{-2r} +e^{-2y} \right) +\Theta\left(e^{2(R-r-y)} \right) -\Theta\left(e^{R-r-y}(e^{-2r}+e^{-2y}) \right)+\Theta\left(e^{-4r} +e^{-4y}\right)\\
&=\,4e^{R-r-y} \left(1+\Theta\left( e^{R-r-y}\right) \right).
\end{align*}
Hence by applying again \eqref{eq:sqrt_approx} we get
$$\theta_r(y)\le 2e^{\frac{R-r-y}{2}}\left(1+\Theta\left(e^{R-r-y}\right) \right)$$
which together with \eqref{eq:theta_lb} concludes the proof.
\end{proof}

Our second lemma gives precise estimates for the measures of several useful combinations of balls. It is heavily used in all the later calculations, and is an important ingredient of our proofs. Note that the claimed formulas look a bit overloaded on the first sight; however, the derived bounds make it applicable for different purposes as we will see in later parts of this paper.
\begin{lemma}\label{lem:intersection_area} For any  \(0 \leq r \leq R\)
  and any \(0 \leq x \leq R\) we have
  \begin{align}
    \label{eq:ball-area}
    \mu\bigl(B_0(x)\bigr) &= e^{-\alpha(R-x)}(1+o(1))\\
    \label{eq:ball-area-capped}
    \mu\bigl(B_r(R)\cap B_0(R)\bigr)&= \frac{2\alpha
      e^{-r/2}}{\pi(\alpha-1/2)} \Bigl(
    1\pm O\bigl(e^{-(\alpha-1/2)r}
    +  e^{-r}\bigr)\Bigr).
  \end{align}
  Further, for \(x \leq R-r\)
  \begin{equation}
  \label{eq:massComp1}
    \mu\bigl((B_r(R) \cap B_0(R))\setminus B_{0}(x)\bigr) =
    \frac{2\alpha e^{-r/2}}{\pi(\alpha-1/2)} \Bigl(
    1\pm O\bigl(e^{-(\alpha-1/2)r}
   +e^{-r}\bigr)\Bigr),
  \end{equation}
  while for \(x \geq R-r\) it holds that
  \begin{multline}
  \label{eq:massComp}
    \mu\bigl((B_r(R) \cap B_0(R))\setminus B_{0}(x)\bigr) =\\
    = \frac{2\alpha
      e^{-r/2}}{\pi(\alpha-1/2)} \left(
    1- \left(1+\frac{\alpha-1/2}{\alpha+1/2}e^{-2\alpha x} \right)e^{-(\alpha-1/2)(R-x)}\right) \Bigl(1
    \pm O\bigl(e^{-r}+e^{-r-(R-x)(\alpha-3/2)}\bigr)
    \Bigr).
  \end{multline}
\end{lemma}
\begin{figure}
  \begin{center}
    \begin{tikzpicture}[dot/.style={draw,circle,inner sep=1pt,fill},scale=0.7]
      \def\rad{2cm}; \def\rpos{1cm};
      \filldraw[fill=gray!20] (0,0) circle (1.4);
      \node[dot] (o) at (0,0) {};
      \draw (0,0) circle (\rad);
      \draw[latex-latex] (o) -- node[pos=0.4,above] {\(R\)} (30:\rad);
      \draw[latex-latex] (o) -- node[below] {\(x\)} (1.4,0);
    \end{tikzpicture}
    \hspace{0.5cm}
    \begin{tikzpicture}[dot/.style={draw,circle,inner sep=1pt,fill},scale=0.7]
      \def\rad{2cm}; \def\rpos{1cm};
      \begin{scope}[even odd rule]
        \clip (0,0) circle (\rad);
        \fill[fill=gray!20] (\rpos,0) circle (\rad);
      \end{scope}
      \node[dot] (c) at (0,0) {};
      \node[dot] (r) at (\rpos,0) {};
      \draw[latex-latex] (c) -- node[pos=0.4,above] {\(R\)} (30:\rad);
      \draw[latex-latex] (c) -- node[below] {\(r\)} (r);
      \draw (c) circle (\rad) (r) circle (\rad);
    \end{tikzpicture}
    \hspace{0.5cm}
    \begin{tikzpicture}[dot/.style={draw,circle,inner sep=1pt,fill},scale=0.7]
      \def\rad{2cm}; \def\rpos{1cm}; \def\radin{0.8cm};
      \begin{scope}[even odd rule]
        \clip (0,0) circle (\rad) (0,0) circle (\radin);
        \fill[fill=gray!20] (\rpos,0) circle (\rad);
      \end{scope}
      \node[dot] (c) at (0,0) {};
      \node[dot] (r) at (\rpos,0) {};
      \draw (c) circle (\rad) (c) circle (\radin) (r) circle (\rad);
      \draw[latex-latex] (c) -- node[pos=0.4,above] {\(R\)} (30:\rad);
      \draw[latex-latex] (c) -- node[below] {\(r\)} (r);
      \draw[latex-latex] (c) -- node[pos=0.3,left=-1pt] {\(x\)} (-120:\radin);
    \end{tikzpicture}
    \hspace{0.5cm}
    \begin{tikzpicture}[dot/.style={draw,circle,inner sep=1pt,fill},scale=0.7]
      \def\rad{2cm}; \def\rpos{1cm}; \def\radin{1.5cm}
      \begin{scope}[even odd rule]
        \clip (0,0) circle (\rad) (0,0) circle (\radin);
        \fill[fill=gray!20] (\rpos,0) circle (\rad);
      \end{scope}
      \node[dot] (c) at (0,0) {};
      \node[dot] (r) at (\rpos,0) {};
      \draw (c) circle (\rad) (c) circle (\radin) (r) circle (\rad);
      \draw[latex-latex] (c) -- node[pos=0.4,above] {\(R\)} (30:\rad);
      \draw[latex-latex] (c) -- node[below] {\(r\)} (r);
      \draw[latex-latex] (c) -- node[pos=0.3,left=-1pt] {\(x\)} (-120:\radin);
    \end{tikzpicture}
  \end{center}
  \caption{\small The grey areas represent the the point sets considered in \eqref{eq:ball-area}, \eqref{eq:ball-area-capped}, \eqref{eq:massComp1} and \eqref{eq:massComp}. Note that the native representation of the hyperbolic space is used, implying that the mass at larger distances from the origin grows exponentially fast.\label{fig:point_set_intersections}}
\end{figure}
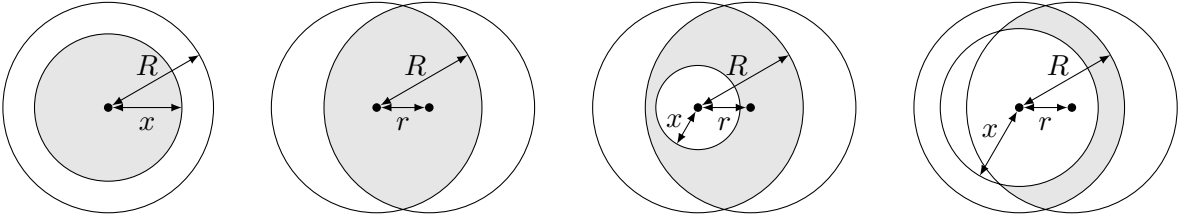
Before we continue with the proof of the lemma, let us give an intuitive description of the statement. Let us in particular consider~\eqref{eq:ball-area-capped}, as the subsequent equations are refinements of it. Equation~\eqref{eq:ball-area-capped} states that the mass of the intersection of $B_r(R)$ and $B_0(R)$ is, up to constants and error terms, equal to $e^{-r/2}$. Recall that in $G_{\alpha, C}(n)$ every point in $B_r(R)\cap B_0(R)$ is connected to the point $p$ with radial coordinate $r$ and $\theta = 0$. Thus, the degree of $p$ is a binomial distribution with parameters $n$ and $e^{-r/2}$. In particular, if $r$ is small, then the expected degree of $p$ is large, and on the other hand, if $r = 2\log n\approx R$, then the expected degree of $p$ is $O(1)$. In other words, the closer a vertex is located to the border of the disc, the smaller its degree will be, and~\eqref{eq:ball-area-capped} allows us to quantify precisely the dependence.
\begin{proof}[Proof of Lemma~\ref{lem:intersection_area}]
The definitions of $\mu$ and $B$, see \eqref{eq:pointset_mu} and \eqref{eq:ball}, imply that
  \begin{align*}
    \mu\bigl(B_0(x)\bigr) = 2\pi\int_0^x f(y) dy =
    \int_0^x\frac{\alpha\sinh(\alpha y)}{\cosh(\alpha R)-1} dy =
    \frac{\cosh(\alpha x) - 1}{\cosh(\alpha R) -1} = (1+o(1))e^{-\alpha(R-x)}.
  \end{align*}
This proves~\eqref{eq:ball-area}. We compute \(\mu\bigl(B_r(R)\cap B_0(R)\bigr)\) as discussed in Equation (\ref{eq:ball_integral}). It follows that
 $$ \mu \bigl( B_r(R) \cap B_0(R) \bigr) = 2\int_0^R\int_0^{\theta_r(y)}f(y)d\theta dy$$
    where $\theta_r(y)$ is as defined in \eqref{eq:def_theta_y}.
   Note that for \(y \leq R-r\) we have \(\theta_r(y) = \pi\) due to the triangle inequality.
   We can therefore split the integral into two parts, and we obtain
  \begin{equation}\label{eq:mu-br-b0-calculation}
    \mu\bigl(B_r(R)\cap B_0(R)\bigr) 
    =
    \mu\bigl(B_0(R-r)\bigr) + 2\int_{R-r}^R \theta_r(y)f(y)dy.
  \end{equation}
  The first part can be computed directly with~\eqref{eq:ball-area}.
  With similar arguments we establish that for
  \(0 \leq x \leq R-r\)
  \begin{equation}\label{eq:mu-br-b0-bx-small-calculation}
    \mu\bigl((B_r(R) \cap B_0(R))\setminus B_{0}(x)\bigr) =
    \mu\bigl(B_0(R-r)\bigr) - \mu\bigl(B_0(x)\bigr) + 2\int_{R-r}^R\theta_r(y)f(y)dy
  \end{equation}
  and for \(R-r \leq x \leq R\)
  \begin{equation}\label{eq:mu-br-b0-bx-large-calculation}
    \mu\bigl((B_r(R) \cap B_0(R))\setminus B_{0}(x)\bigr) = 2\int_{x}^R\theta_r(y)f(y)dy.
  \end{equation}
In the sequel we use Lemma~\ref{lem:theta} to compute the integrals in \eqref{eq:mu-br-b0-calculation},
  \eqref{eq:mu-br-b0-bx-small-calculation} and
  \eqref{eq:mu-br-b0-bx-large-calculation}. We begin with 
  \begin{equation}\label{eq:int-theta-f-dy}
      2\int_{x}^R \theta_r(y)f(y)dy = 2\int_{x}^R 2e^{\frac{R-r-y}{2}}\left(1 \pm
      O(e^{R-r-y})\right)\frac{\alpha \sinh(\alpha y)}{2\pi(\cosh(\alpha R)-1)}dy
  \end{equation}
  where we assume that the lower
  bound \(x\) is such that \(x \geq R-r\).
  We first solve the integral for the leading term without the error term. A simple calculation shows that
  \[
  	\int e^{-y/2} \sinh(\alpha y) dy = \frac{2}{4\alpha^2-1}e^{-y/2}(2 \alpha \cosh(\alpha y) + \sinh(\alpha y)).
  \]
Thus,
  \begin{multline}\label{eq:int-theta-f-dy-solved}
    2\int_{x}^R 2e^{\frac{R-r-y}{2}}\frac{\alpha \sinh(\alpha
      y)}{2\pi(\cosh(\alpha R)-1)}dy =\\= \frac{4 \alpha e^{-r/2} \bigl(2 \alpha \cosh(\alpha R)+\sinh(\alpha R)-e^{\frac{R-x}{2}} (2 \alpha \cosh(\alpha x)+\sinh(\alpha x)\bigr)}{\pi(4 \alpha^2-1)(\cosh(\alpha R)-1)}.
  \end{multline}
  Expanding all trigonometric terms to their definition as sums of exponential
  functions and the fact that \(1/(\cosh(\alpha R) -1) = 2\exp(-\alpha
  R)(1+\Theta(e^{-\alpha R}))\) we obtain
	\begin{equation}
	\label{eq:int-theta-f-dy-solved2}
  \begin{aligned}
    \eqref{eq:int-theta-f-dy-solved} &= \frac{8\alpha
      e^{-r/2}}{\pi(4\alpha^2-1)}\Bigl[\left(\alpha +
    \tfrac{1}{2}\right)\left(1-e^{-(\alpha-1/2)(R-x)}\right) \\ &\hspace{25mm}+ (\alpha -
    \tfrac{1}{2})\left(e^{-2\alpha
      R}-e^{(R-x)/2-\alpha(R+x)}\right)\Bigr](1+\Theta(e^{-\alpha R})) \\&=
    \frac{2\alpha e^{-r/2}}{\pi(\alpha-1/2)}\left(1- \left(1+\frac{\alpha-1/2}{\alpha+1/2}e^{-2\alpha x} \right)
    e^{-(\alpha-1/2)(R-x)}\right)\left(1+\Theta\left(e^{-\alpha R}\right)\right) .
  \end{aligned}
  \end{equation}
  The integral over the error term in \eqref{eq:int-theta-f-dy} is at most
  \begin{equation*}
    \label{eq:eq-int-theta-f-dy-error-solved}
    \int_{x}^R O\left(e^{\frac{3(R-r-y)}{2}}\right)\frac{\sinh(\alpha y)}{\cosh(\alpha R)}dy
    = \int_{x}^R O\left(e^{\frac{3(R-r-y)}{2} + \alpha(y - R)}\right)dy
    =O\Bigl(e^{-3r/2} + e^{-3r/2-(R-x)(\alpha-3/2)}\Bigr).
  \end{equation*}
  Combining the above with Equation~\eqref{eq:int-theta-f-dy-solved2} we
  finally have
  \begin{multline}
    \label{eq:int-theta-f-dy-solved-with-error}
    2\int_{x}^R \theta_r(y)f(y)dy = \\
    =\frac{2\alpha
      e^{-r/2}}{\pi(\alpha-1/2)} \left(
    1- \left(1+\frac{\alpha-1/2}{\alpha+1/2}e^{-2\alpha x} \right)e^{-(\alpha-1/2)(R-x)}\right)
    \Bigl(1 \pm O\bigl(e^{-r}+e^{-r-(R-x)(\alpha-3/2)}\bigr)
    \Bigr).
  \end{multline}
  For \(x = R-r\) we obtain
  \begin{equation*}
    2\int_{R-r}^R \theta_r(y)f(y)dy = \frac{2\alpha
      e^{-r/2}}{\pi(\alpha-1/2)} \Bigl(
    1\pm O\bigl(e^{-(\alpha-1/2)r}+e^{-r}\bigr)
    \Bigr)
  \end{equation*}
which results in 
  \begin{equation*}
    \mu\bigl(B_r(R)\cap B_0(R)\bigr) = e^{-\alpha r}(1+o(1)) + \frac{2\alpha
      e^{-r/2}}{\pi(\alpha-1/2)} \Bigl(
    1\pm O\bigl(e^{-(\alpha-1/2)r}+e^{-r}\bigr)
    \Bigr).
  \end{equation*}
  Note that the term \(e^{-\alpha r}\) can be written as \(e^{-r/2 -
    (\alpha-1/2)r}\) and can be incorporated into the first error term,
  proving our claim.

  For \(\mu\bigl((B_r(R) \cap B_0(R))\setminus B_{0}(x)\bigr)\) and \(0 \leq x
  \leq R-r\) we obtain the same bound as in this regime
  \(\mu\bigl(B_0(R-r)\bigr)\) asymptotically dominates
  \(\mu\bigl(B_0(x)\bigr)\). The solution for \(R-r \leq x \leq R\) is given
  by Equation~\eqref{eq:int-theta-f-dy-solved-with-error}.

\end{proof}

The next lemma confirms the somewhat intuitive fact that a ball around a point has a higher measure the closer the point is located to the center of the disk.
\begin{lemma}
\label{lem:monotone}
For all $0\le r_0\le R$, all $r_0\le r \le R$ and all $0\le x \le R$
$$\mu(B_0(R)\cap B_{r_0}(R))\geq \mu(B_0(R)\cap B_r(R))$$
and 
$$\mu(B_0(R)\cap B_{r_0}(R)\setminus B_0(x))\geq \mu(B_0(R)\cap B_r(R)\setminus B_0(x)).$$
\end{lemma}
\begin{proof}
Because of Equation (\ref{eq:mu-br-b0-calculation}), (\ref{eq:mu-br-b0-bx-small-calculation}) and (\ref{eq:mu-br-b0-bx-large-calculation}) it suffices to show that $\theta_{r_0}(y)\geq \theta_{r}(y)$ for $0\le y \le R$. 
To see this, recall first that $\theta_r(y)$ is given by the solution of 
$$\theta_r(y)=\arccos\left(\frac{\cosh(r)\cosh(y)-\cosh(R)}{\sinh(r)\sinh(y)} \right).$$
The claim follows, since $\frac{\cosh(r)\cosh(y)-\cosh(R)}{\sinh(r)\sinh(y)}$ is increasing in $r$ and $\arccos$ is decreasing.
\end{proof}

\section{Proofs of the Main Results}

Before we give the proofs for our theorems, let us briefly describe a technique which we use to show concentration for the clustering coefficient and the degree sequence. We wish to apply an Azuma-Hoeffding-type large deviation inequality (see Lemma~\ref{lem:concentration} below) to show that the sum of the local clustering coefficients $X:=\sum_{v\in V}c_v$ and the number of vertices of degree $k$ $D_k$ are concentrated around its expectation.

In a typical setting, such concentration inequalities require some kind of \emph{Lipschitz condition} that is satisfied by the function under consideration. In our specific setting, the functions are $X$ and $D_k$, and it is required to provide a bound for the maximum effect that any vertex has. However, the only a priori bound that can be guaranteed is that for example the number of vertices of degree $k$ can change by at most $n$, as a vertex may connect or not connect to any other vertex. To make the situation worse, this bound is even tight, since a vertex can be placed at the center of disc, i.e., if it has radial coordinate equal to 0.

We will overcome this obstacle as follows. Instead of counting the total number of vertices of degree $k$ and the sum of all the local clustering coefficients, we will consider only vertices that lie far away from the center of the disc, i.e., which have radial coordinate larger than $\beta R$, for some appropriate $\beta >0$. Moreover, we will consider only vertices such that all their neighbors have a large radial coordinate as well. This restriction will allow us to bound the maximum effect on the target function, as with high probability all these vertices do not have too large degree.

More formally, we proceed as follows. We partition the vertex set of $G_{\alpha, C}(n)$ into two sets. The \emph{inner set} $I = I(\beta)$ contains all vertices of radius at most $\beta R$ while the \emph{outer set} $O = O(\beta)$ contains all vertices of radius larger than $\beta R$. 

We will use the following large deviation inequality. Let $f$ be a function on the random variables $X_1, \dots, X_n$ that take values in some set $A_i$. We say that $f$ is Lipschitz with coefficients $c_1\dots c_n$ and bad event $\mathcal{B}$ if for all $x,y\in A$ 
$$\left|\mathbb{E}[f|X_{1},\dots X_{i-1}, X_i=x, \overline{\mathcal{B}}]-\mathbb{E}[f|X_{1},\dots X_{i-1}, X_i=y, \overline{\mathcal{B}}]\right|\le c_i .$$
(We denote by $\overline{\mathcal{B}}$ the complement of $\mathcal{B}$.)
Then the following estimates are true.
\begin{theorem}[Theorem~7.1 in \cite{dubhashi2009concentration}]
\label{lem:concentration}
Let $f$ be a function of $n$ independent random variables $X_1,\dots, X_n$, each $X_i$ taking values in a set $A_i$, such that $\mathbb{E}[f]$ is bounded. Assume that 
$$m\le f(X_1, \dots, X_n)\le M.$$
Let $\mathcal{B}$ any event, and let $c_i$ be the maximum effect of $f$ assuming the complement $\overline{\mathcal{B}}$ of $\mathcal{B}$:
$$\max_{x,y}|\mathbb{E}[f|X_1,\dots, X_{i-1}, X_i=x, \overline{\mathcal{B}}]-\mathbb{E}[f|X_1,\dots, X_{i-1}, X_i=y, \overline{\mathcal{B}}] | \le c_i.$$
 Then
$$\Pr[f> \mathbb{E}[f]+t+(M-m)\Pr[\mathcal{B}]] \le e^{-{2t^2}/{\sum_i c_i^2}} +\Pr[\mathcal{B}]$$
and
$$\Pr[f< \mathbb{E}[f]-t-(M-m)\Pr[\mathcal{B}]] \le e^{-{t^2}/{\sum_i c_i^2}} +\Pr[\mathcal{B}].$$
\end{theorem}
In our setting, the random variables of interest are usually functions of $X_1, \dots , X_n$, where $X_i$ denotes the coordinates of the $i$th vertex. For the clustering coefficient and the degree sequence a coordinate change can not have a large effect on the random variable as long as the degree of the corresponding vertex is small. The following lemma says that in $O(\beta)$ the degrees of the vertices is bounded with high probability. 
\begin{lemma}
\label{lem:lipschitz}
Let $\alpha > 1/2$ and $0<\beta<1$. There is a constant~$c>0$ such that the probability for the bad event
$$\mathcal{B}:=\left\{\text{there is a vertex in~$O(\beta) = B_0(R)\setminus B_0(\beta R)$ with degree at least~$cn^{1-\beta}$}\right\}$$
is at most $\Pr[\mathcal{B}]=e^{-\Omega\left( n^{1-\beta}\right)}.$
\end{lemma}
\begin{proof}
Note that unless~$\mathcal{B}$ holds every vertex in~$O:=O(\beta)$ is connected to at most~$cn^{1-\beta}$ other vertices in~$O$. Therefore, any change in~$X_i$ can increase or decrease the number of vertices in~$O$ of degree~$k$ by at most~$cn^{1-\beta}+1$ (the additional ``+1'' is due to the fact that vertex $i$ could change its degree as well). It remains to bound the probability of~$\mathcal{B}$.
By applying Lemma~\ref{lem:monotone} we see that the expected degree in $O$ of a vertex with radius $r\geq \beta R$ is at most
$$n\cdot \mu(B_0(R)\cap B_{\beta R}(R)\setminus B_0(\beta R))\stackrel{\textrm{(Lem.\ \ref{lem:intersection_area})}}{=}O\left(n\cdot e^{-\frac{\beta R}{2}} \right) =O(n^{1-\beta}) .$$
Hence the expected degree of a vertex of radius at least $\beta R$ is at most $c'n^{1-\beta}$ for some constant $c'$. For $c:=2ec'$ it suffices to apply a Chernoff bound to show that for a vertex $v$ of radius at least $\beta R$
$$\Pr[d_O(v)>cn^{1-\beta}] \le 2^{-cn^{1-\beta}} .$$
The statement of the lemma follows by union bound over all vertices. 
\end{proof}

\subsection{The Clustering Coefficient}
Recall the definition of the local and global clustering coefficient in \eqref{eq:def_local_clustering} and \eqref{eq:def_global_clustering}. We will need the following technical statement, which gives an estimate for the measure of the intersection of the balls around two coordinates $(r_1, \theta_1)$ and $(r_2, \theta_2)$ if their angle difference $\theta:=|\theta_1-\theta_2|$ is very small. 

\begin{lemma}
\label{lem:double_intersection}
Let  $\beta>1/2$, $\beta R \le x\le R$, $r_1\geq r_2\geq x$ and $0\le \theta \le e^{-r_2/2}-e^{-r_1/2}$. Then
\begin{equation}
\mu\left(  B_0(R)\cap B_{r_1, 0}(R)\cap B_{r_2, \theta}(R)\setminus B_0(x)   \right)  =\mu\left(B_0(R)\cap B_{r_1}(R)\setminus B_0(x)\right).
\end{equation}
\end{lemma}
\begin{proof}
It follows from similar observations like the ones that lead to \eqref{eq:ball_integral} that 
$$\mu\left(B_0(R)\cap B_{r_1, 0}(R)\cap B_{r_2, \theta}(R)\setminus B_0(x) \right)=\int_x^{R}\int_{\max\{-\theta_{r_1}(y), \theta-\theta_{r_2}(y)\}}^{\min\{  \theta_{r_1}(y), \theta+\theta_{r_2}(y)\}}f(y)d\phi dy.$$
For $ \theta \le e^{-r_2/2}-e^{-r_1/2}$ and $r_1\geq r_2$, using Lemma~\ref{lem:theta}, it can be verified that
$$\max\{-\theta_{r_1}(y), \theta-\theta_{r_2}(y)\}= -\theta_{r_1}(y) \quad \text{ and } \quad \min\{\theta_{r_1}(y), \theta+\theta_{r_2}(y)\}=\theta_{r_1}(y).$$
\qed
\end{proof}

\begin{proof}[Theorem~\ref{thm:clustering}]
Let $\beta:=2/3$ and for a graph $G=(V, E)$ let
$$X:=\sum_{\stackrel{v\in V}{ deg(v)\geq 2}} \frac{\left|\{\{u_1, u_2\}\in E~|~u_1, u_2\in \Gamma(v)\} \right|}{\binom{deg(v)}{2}}$$ 
and 
$$Y:=\sum_{\stackrel{v\in O(\beta)}{ deg(v)\geq 2}} \frac{\left|\{\{u_1, u_2\}\in E~|~u_1, u_2\in \Gamma(v)\cap O(\beta)\} \right|}{\binom{deg(v)}{2}}.$$
Clearly, $\overline{c}=\frac{X}{n}$ and $X\geq Y$.
% and note that on order to derive a lower bound on the global clustering coefficient it suffices to consider only a subset $S\subseteq V$ of the vertices and calculate their contribution $\frac{1}{n}\sum_{v\in S}\bar{c}_v$. 
It therefore suffices to derive a constant lower bound on $\mathbb{E}[Y]$ and to show that $Y$ is concentrated around its expectation. Let $\mathbb{E}[Y_r~|~\mathcal{E}]$ be the expected value of $\frac{\left|\{\{u_1, u_2\}\in E~|~ u_1, u_2\in \Gamma(v)\cap O(\beta)\} \right|}{\binom{deg(v)}{2}}$ for a vertex $v$ with radius $r$ conditioned on the event $\mathcal{E}$ that the vertex has degree at least $2$. We observe that $\mathbb{E}[Y_r~|~\mathcal{E}]$ is exactly the probability that two randomly chosen neighbors $u_1$ and $u_2$ of a vertex $v$ at radius $r$ are connected. 
In order to derive this probability for a fixed vertex $v$ at radius $r$, let us suppose that $u_1$ is at coordinate $(y, \phi)\in B_0(R)\cap B_r(R)\setminus B_0(\beta R)$. Note that by \eqref{eq:def_theta_y} these coordinates satisfy $\beta R\le y \le R$ and $-\theta_r(y)\le \phi \le \theta_r(y)$. Moreover, the probability for the event that $u_1$ is at $(y, \phi)$ is given by
$$\frac{f(y)}{\mu(B_0(R)\cap B_r(R))}.$$ 
The vertex $u_2$ is connected to $u_1$ in such a way that $\{u_1, u_2\}$ contributes to $Y_r$ only if $u_2$ lies in the intersection of the balls $B_0(R)\cap B_r(R)\setminus B_0(\beta R)$ and $B_0(R)\cap B_y(R)\setminus B_0(\beta R)$. Therefore, the contribution of $u_2$ to $\mathbb{E}[Y_r~|~\mathcal{E}]$, given the coordinates of $u_2$, is
$$\frac{\mu(B_{r,0}(R)\cap B_{y, \phi}(R)\cap B_0(R)\setminus B_0(\beta R) )}{\mu(B_0(R))\cap B_r(R))}.$$
Note that we choose $u_2$ uniformly from all neighbors of $v$ which makes it possible that $u_1=u_2$. However, since the degree of $v$ is at least $2$ this event happens with probability at most $1/2$.  
Putting all the above facts together implies that
\begin{align*}
\mathbb{E}[Y_r~|~\mathcal{E}]&\geq\frac{1}{2}\int_{\beta R}^R \int_{-\theta_r(y)}^{\theta_r(y)}  \frac{f(y)\mu(B_{r,0}(R)\cap B_{y, \phi}(R)\cap B_0(R)\setminus B_0(\beta R) )}{(\mu(B_0(R)\cap B_r(R)))^2} d\phi dy.
\end{align*}
Since the term in the integral above does not depend on the angle, we can replace the integral from $-\theta_r(y)$ to $\theta_r(y)$ by twice the integral from $0$ to $\theta_r(y)$. Further, we derive a lower bound on that term by integrating the radius only from $r$ to $R$. Observe also that for $r\le y \le R$ the upper boundary of the angle, $\theta_r(y)\stackrel{(\text{Lem.~\ref{lem:theta}})}{=}(1+o(1))2e^{\frac{R-r-y}{2}}$, is at least $\xi:=e^{-\frac{r}{2}}-e^{-\frac{y}{2}}$ and that therefore the term above is at least
%&\geq \frac{1}{(\mu(b_0(R)\cap B_r(R)\setminus B_0(\beta R)))^2} \left( \int_{\beta R}^r\int_0^{e^{-\frac{y}{2}}-e^{-\frac{r}{2}}}f(y)\cdot \mu(B_{r,0}(R)\cap B_{y, \phi}(R)\cap B_0(R)\setminus B_0(\beta R) ) d\phi dy \right.\\
%& \left. +  \int_r^R\int_0^{e^{-\frac{r}{2}}-e^{-\frac{y}{2}}}f(y)\cdot \mu(B_{r,0}(R)\cap B_{y, \phi}(R)\cap B_0(R)\setminus B_0(\beta R) ) d\phi dy  \right)\\
%& \stackrel{Lem.~\ref{lem:double_intersection}}{\geq} 
 \begin{align*}
 &\frac{ \int_r^R\int_0^{\xi}f(y)\cdot \mu(B_{r,0}(R)\cap B_{y, \phi}(R)\cap B_0(R)\setminus B_0(\beta R) ) d\phi dy}{(\mu(B_0(R)\cap B_r(R)\setminus B_0(\beta R)))^2} .
 \end{align*}
 Applying Lemma~\ref{lem:double_intersection}, this integral simplifies to
 \begin{align*}
& \frac{1}{(\mu(B_0(R)\cap B_r(R)\setminus B_0(\beta R)))^2} \int_r^R\int_0^{\xi}f(y)\cdot \mu(B_0(R)\cap B_y(R)\setminus B_0(\beta R)) d\phi dy\\
&\hspace{2cm}\stackrel{\mathclap{\text{(Lem.~\ref{lem:intersection_area}), \eqref{eq:approx_cosh}}}}{\geq}\quad\quad \frac{\alpha-1/2}{12}e^{r-\alpha R} \int_r^{R}(e^{-r/2}-e^{-y/2})e^{y(\alpha-1/2)}d\phi dy\\
&\hspace{2cm} =  \frac{1}{12} \left( e^{r/2-\alpha R} \left[e^{y(\alpha-1/2)} \right]_r^R -\frac{\alpha-1/2}{(\alpha-1)}e^{r-\alpha R}\left[e^{y(\alpha-1)}\right]_r^R\right)\\
&\hspace{2cm}\geq  \frac{1}{24}\left(e^{-(R-r)/2} +\frac{1}{2(\alpha-1)}e^{-\alpha(R-r)}-\frac{\alpha-1/2}{\alpha-1}e^{-(R-r)} \right).
\end{align*}
Let $v_R$ be a vertex at radius $R$. It follows from Lemma~\ref{lem:monotone} that the degree distribution of every vertex in the graph dominates the degree distribution of $v_R$. Therefore, for any $v\in V$
\begin{align*}
\Pr[\mathcal{E}]&=\Pr[deg(v)\geq 2]\geq \Pr[deg(v_R)\geq 2]\geq \Pr[deg(v_R)=2]\\
&= \binom{n-1}{2}(\mu(B_0(R)\cap B_R(R)))^2 (1-\mu(B_0(R)\cap B_R(R)))^{n-3}\\
& \geq n^2 e^{-R}\frac{\alpha^2}{\pi^2(\alpha-1/2)^2}e^{-\frac{2\alpha e^{-R/2}}{\pi(\alpha-1/2)}n}=e^{-C}\frac{\alpha^2}{\pi^2(\alpha-1/2)^2}e^{-\frac{2\alpha e^{-C/2}}{\pi(\alpha-1/2)}}.
\end{align*}
By integrating $\mathbb{E}[Y_r~|~\mathcal{E}]$ over all $\beta R \le r \le R$ and multiplying with $\Pr[\mathcal{E}]$, $f(y)$ and $n$ we get the expected value of $Y$ 
\begin{align}
\label{eq:global_clustering}
\mathbb{E}[Y]&= n \int_{\beta R}^R f(y)\cdot \Pr[\mathcal{E}]\cdot \mathbb{E}[\overline{c_r}~|~\mathcal{E}] dr\geq \frac{n\cdot e^{-C} \alpha^2e^{-\frac{2\alpha e^{-C/2}}{\pi(\alpha-1/2)}}}{600\pi^3(\alpha-1/2)(\alpha+1)(\alpha+1/2)}.
\end{align}
Thus, $\mathbb{E}[Y] =\Theta(n).$

Set $f:= Y, t:=n^{6/7}$ and `bad' event $\mathcal{B}$ and $c$ as stated in Lemma~\ref{lem:lipschitz}. Note that each coordinate change can influence $f$ by at most $c_i:=cn^{1-\beta}+1$ as long as $\mathcal{\bar{B}}$ holds. It therefore follows from Theorem~\ref{lem:concentration} and  Lemma~\ref{lem:lipschitz} that $\Pr\left[X\le \mathbb{E}[Y]-n^{6/7}-Pr[\mathcal{B}]\right]=o(1)$ and therefore that the clustering  coefficient is with high probability at least $\mathbb{E}[Y]/n=\Theta(1)$.
\qed
\end{proof}

\subsection{Vertices of Small Degree }

In this section we prove the first part of Theorem~\ref{thm:degree_seq}. Before we give all technical details, let us describe briefly the main proof idea. Given the estimates in the previous sections, in particular Lemma~\ref{lem:intersection_area}, it is conceptually not very difficult to compute the expected number of vertices of degree $k$ in $G_{\alpha, C}(n)$. Nevertheless, it is not clear how to show the claimed strong concentration bound. To this end, we will apply the same large deviation inequality as in the previous section. 

Recall our partitioning of the vertex set of $G_{\alpha, C}(n)$ into the \emph{inner set} $I = I(\beta)$ and the \emph{outer set} $O = O(\beta)$. Moreover, let $e(I,O)$ count the number of edges with one endpoint in $I$ and the other in $O$. The next two lemmas show that $e(I,O)$ and $|I|$ are small. This indicates that most of the vertices of degree $k$, for not too large $k$, will lie in $O$. Let $D_k$ denote the number of vertices in $O$ which have degree $k$ in $O$, i.e., set 
$$D_k(\beta) = \left|\big\{v\in O(\beta) ~\big|~ |N(v) \cap O(\beta)| = k \big\}\right|.$$
In Lemma~\ref{lem:degree_sequence} we derive the expectation of $D_k(\beta)$ and finally, we combine everything to show that $D_k$ is tightly concentrated around its expectation. 

\begin{lemma}
\label{lem:few_vertices_I}
In $G_{\alpha, C}(n)$, with  probability at least $1 - e^{-n^{\Omega(1)}}$
$$|I(\beta)| \le \max\{n^{1/2}, 4en^{1-2\alpha (1-\beta)}e^{-\alpha C(1-\beta)}\}.$$
\end{lemma}
\begin{proof}
The number of vertices in $I$ is distributed like $\textrm{Bin}(n, \mu\bigl(B_0(\beta R)\bigr))$. Therefore the expected number of vertices in $I$ is (by  Lemma~\ref{lem:intersection_area}) bounded by 
$$\mathbb{E}[|I(\beta)|]\le  n \mu\bigl(B_0(\beta R)\bigr) = (1+o(1))ne^{-\alpha(R-\beta R)}\le 2n^{1-2\alpha (1-\beta)}e^{-\alpha C(1-\beta)}$$
and it follows by the Chernoff bound that for $t=\max\{n^{1/2}, 4en^{1-2\alpha (1-\beta)}e^{-\alpha C(1-\beta)}\}$
$$\Pr[|I(\beta)| >t]\le 2^{-t}=e^{-n^{\Omega(1)}}.$$
\end{proof}

\begin{lemma}
\label{lem:few_edges_IO}
Let $\varepsilon>0$ and $0<\beta<1$. Then it holds with high probability that
$$e(I(\beta),O(\beta))= O\left( n^{1-(2\alpha-1)(1-\beta)}\log n\right).$$
\end{lemma}

\begin{proof}
Let $r_0:=\left(1-\frac{1}{2\alpha}\right)\log n+C$ and note that by Lemma~\ref{lem:intersection_area} the expected number of vertices of radius at most $r_0$ is at most
$$n\mu(B_0(r_0))=(1+o(1))ne^{-\alpha(R-r_0)}=ne^{-(\alpha+1/2)\log n} =o(1).$$
We therefore have with high probability no vertex of radius at most $r_0$. For all $r_0\le r\le \beta R$ we get by Lemma~\ref{lem:intersection_area} that the expected degree in $O$ of a vertex at radius $r$ is at most
\begin{equation}
\label{eq:expected_o_deg_of_r}
n\mu(B_r(R)\cap B_0(R)\setminus B_0(\beta R)) = O\left(ne^{-r/2}\right).
\end{equation}
We now integrate over \eqref{eq:expected_o_deg_of_r} in $I$ to bound the expected number of edges between $I$ and $O$. Recall that there are with high probability no vertices of radius at most $r_0$ and it therefore suffices to integrate from $r_0$ to $\beta R$ 
\begin{align*}O\left(n^2\int_{r_0}^{\beta R} e^{-r/2}p(r)dr \right) &= O\left(n^{2}e^{-\alpha R}\int_0^{\beta R}e^{(\alpha-1/2)r}dr \right)=O\left(n^{2-2\alpha+2\beta(\alpha -1/2)} \right)\\
&=O\left( n^{1-(2\alpha-1)(1-\beta)}\right).
\end{align*}
The lemma follows by Markov's inequality. 

\end{proof}

In the next lemma we derive the expected degree sequence of the subgraph spanned by $O$.

\begin{lemma}
\label{lem:degree_sequence}
Let $\alpha > 1/2, C\in \mathbb{R}$, and $\max\left\{3/5, 1/(2\alpha)\right\} < \beta <1$. Set $\delta:=\min\{2(2\beta-1), 1/2\}$. Then, for all $0 \le k = o(n^\delta)$ we have that
$$\mathbb{E}[D_k(\beta)]=  \bigl(1+o(1)\bigr)\frac{2n\alpha e^{-\alpha C}}{k!} \Bigl(\frac{2\alpha}{\pi(\alpha-1/2)}\Bigr)^{2\alpha}\left( \Gamma(k-2\alpha)-\int_{0}^{\xi} t^{k-2\alpha-1}e^{-t}dt  \right),$$
where
$
	\xi = \frac{2\alpha}{\pi(\alpha - 1/2)} e^{-C/2}.
$
\end{lemma}

\begin{proof}
  Let \(p = (r, \theta)\) be an arbitrary fixed vertex with \(r > \beta R\).
  Denote with \(\bar{q}_r\) the probability that a random vertex \(p' = (r',
  \theta')\) has radius at least \(\beta R\) and distance at most \(R\) from
  \(p\). We have that
  \begin{equation*}
    \bar q_r = \mu\bigl((B_r(R) \cap B_0(R))\setminus B_{0}(\beta R)\bigr) .
  \end{equation*}
  The probability that \(p\) has \(k\) neighbors with radius larger than $\beta R$ then corresponds to
  the probability that a binomial random variable \(\Bin(n-1,\bar{q}_r)\) has
  value \(k\). Therefore the expected value of \(D_k\) can be computed by
  \begin{equation}
    \label{eq:deg_seq_exp}
    \mathbb{E}[D_k] =  n\int_{\beta R}^R
    \binom{n-1}{k}\bar{q}_r^k(1-\bar{q}_r)^{n-1-k} p(r) dr .
  \end{equation}
  By Lemma \ref{lem:intersection_area}, Equation~\eqref{eq:massComp}, and the observation that $\Theta(e^{-2\alpha \beta R})=O(e^{-r})$ it follows that
  \begin{equation*}
    \bar q_r = \frac{2\alpha e^{-r/2}}{\pi(\alpha-1/2)} \Bigl(
    1-e^{-(\alpha-1/2)(1-\beta)R}\Bigr)\Bigl(1 \pm O\bigl(e^{-r}+e^{-r-(\alpha-3/2)(1-\beta)R}\bigr)
    \Bigr).
  \end{equation*}
  Recall that \(R = 2\log n+C\) and
  therefore \(e^{-r} \leq e^{-\beta R} = o(1/n)\).  
  To approximate the integral in Equation~\eqref{eq:deg_seq_exp} we separate
  the main and error terms as follows. Write
  \begin{equation*}
    q_r = \frac{2\alpha e^{-r/2}}{\pi(\alpha-1/2)} \Bigl(
    1-e^{-(\alpha-1/2)(1-\beta)R}\Bigr) \quad\text{and}\quad
    f_r = O\bigl(e^{-r-(\alpha-3/2)(1-\beta)R}\bigr)+o(1/n) .
  \end{equation*}
	Note that $\bar q_r = q_r (1 \pm f_r) $.
	Before we proceed with the estimation of the expression in~\eqref{eq:deg_seq_exp} let us prove some auxiliary facts. For all~\(r > \beta R\) and all~\(\alpha > 1/2\) we claim to have the following
  properties for \(q_r\), \(f_r\) and \(k\):
\begin{equation}
\label{eq:aux_prop}
  \textrm{i)}~q_r f_r = o(1/n),
  \qquad\textrm{ii)}~(q_r)^2 = o(1/n),
  \qquad\textrm{ and \qquad iii)}~f_r k = o(1).
\end{equation}
  Property i) is established by observing that whenever $\beta > 3/5$
  \begin{equation*}
    q_rf_r = O\bigl(e^{-(\alpha-3/2)(1-\beta)R-r/2-r}\bigr)
    \stackrel{(r \ge \beta R)}{=} O\bigl(e^{-(\alpha-3/2)(1-\beta)R-{3\beta R}/{2}}\bigr)
    = O\bigl(n^{3 - 6\beta - 2\alpha + 2\alpha\beta}\bigr) = o(1/n).
  \end{equation*}
  The second claimed property is true because \((q_r)^2 = O(e^{-\beta R}) = O(n^{-2\beta})=O(n^{-6/5})\),
  while the third is satisfied by our choice of \(k=o(n^{2(2\beta-1)})\). We claim that these facts imply
  \begin{equation}\label{eq:degseq-approx-integral}
    n\int_{\beta R}^R
    \binom{n-1}{k}(\bar{q}_r)^k(1-\bar{q}_r)^{n-1-k} p(r) dr =
    \bigl(1+o(1)\bigr)\frac{n}{k!}
    \int_{\beta R}^R
    (n q_r)^ke^{-n q_r} p(r) dr.
  \end{equation}
  Consider first the binomial coefficient. Since $k \ll \sqrt{n}$ we have
  \begin{equation*}
    \binom{n-1}{k} = \frac{(n-1)^k}{k!} \, \prod_{i=0}^{k-1}\left(1 - \frac{i}{n-1}\right) = \bigl(1+o(1)\bigr)\frac{n^{k}}{k!}.
  \end{equation*}
  Next we consider the term \((\bar q_r)^k\). By using the third property in~\eqref{eq:aux_prop} we obtain that
  \begin{equation*}
    (\bar q_r)^k = (q_r(1+f_r))^k = (q_r)^k\bigl(1+\Theta(f_r k)\bigr) =
    (1+o(1))(q_r)^k.
  \end{equation*}
  Finally, we have that
  \begin{equation*}
    \begin{split}
      (1-\bar q_r)^{n-1-k} &~~=~~ (1-q_r)^{n-1-k}\bigl(1 +
      q_rf_r/(1-q_r)\bigr)^{n-1-k} \\
      &~~=~~ e^{(n-1-k)\log(1-q_r)}(1+\Theta(n q_rf_r))\\
      &\stackrel{\eqref{eq:aux_prop}, i)}{=} e^{-n q_r+O((n-1-k)q_r^2)+(k+1)q_r}\bigl(1+o(1)\bigr)\\
      &\stackrel{\eqref{eq:aux_prop}, ii)}{=} e^{-n q_r}(1+o(1)).
    \end{split}
  \end{equation*}
  This completes the proof of~\eqref{eq:degseq-approx-integral}. We now estimate the integral on the right hand side of
  Equation~\eqref{eq:degseq-approx-integral}. To do so we first perform a
  variable transformation \(t = nq_r\). Then, the density \(p(r)\)
  can be expressed as
  \begin{equation*}
    \begin{split}
      p(r) &= \frac{\alpha\sinh(\alpha r)}{\cosh(\alpha R)-1} =
      \frac{\alpha}{2(\cosh(\alpha R)-1)}\bigl(e^{\alpha r}- e^{-\alpha
        r}\bigr)\\
      &= \frac{\alpha}{2(\cosh(\alpha R)-1)}
      \left((nq_r)^{-2\alpha}n^{2\alpha}
      \left(\frac{2\alpha \left(
    1-e^{-(\alpha-1/2)(1-\beta)R}\right) }{\pi(\alpha-1/2)}\right)^{2\alpha}-o(1)\right)\\
  \end{split}
  \end{equation*}
Since \(1/(\cosh(\alpha R)-1) =
  2n^{-2\alpha}e^{-\alpha C}\bigl(1+o(1/n)\bigr)\), the above calculation yields that
\[
	p(r) = \alpha e^{-\alpha C} \Bigl(\frac{2\alpha}{\pi(\alpha-1/2)}\Bigr)^{2\alpha}
      (nq_r)^{-2\alpha}\bigl(1+o(1)\bigr).
\]
Further we have that
  \begin{equation*}
    dt = n\frac{2\alpha\left(
    1-e^{-(\alpha-1/2)(1-\beta)R}\right)}{\pi(\alpha-1/2)}(-1/2)e^{-r/2}dr =
    -\frac{1}{2}t dr\quad\Rightarrow\quad dr = -2t^{-1}dt,
  \end{equation*}
  and for the upper and lower bounds that
  \begin{equation*}
    \begin{split}
      r_0 = \beta R \quad&\longrightarrow\quad t_0 = n
      \frac{2\alpha\left(
    1-e^{-(\alpha-1/2)(1-\beta)R}\right)}{\pi(\alpha-1/2)} e^{-\beta R/2} =
      \frac{2\alpha\left(
    1-e^{-(\alpha-1/2)(1-\beta)R}\right)}{\pi(\alpha-1/2)} e^{-\frac{\beta C}{2}}n^{1-\beta}\\
      r_1 = R \quad&\longrightarrow\quad t_1 = n
      \frac{2\alpha\left(
    1-e^{-(\alpha-1/2)(1-\beta)R}\right)}{\pi(\alpha-1/2)} e^{-R/2} =
      \frac{2\alpha\left(1-e^{-(\alpha-1/2)(1-\beta)R}\right)}{\pi(\alpha-1/2)}e^{-C/2}.
    \end{split}
  \end{equation*}
  Putting everything together we have that the right hand side of~(\ref{eq:degseq-approx-integral}) can be approximated by
  \begin{multline*}
    \bigl(1+o(1)\bigr)
    \frac{2n\alpha e^{-\alpha C}}{k!}
    \Bigl(\frac{2\alpha}{\pi(\alpha-1/2)}\Bigr)^{2\alpha}
    \int_{t_1}^{\infty}
    t^{k-2\alpha-1}e^{-t}dt \\
%  =    \bigl(1+o(1)\bigr)
%    \frac{2n\alpha e^{-\alpha C}}{k!}
%    \Bigl(\frac{2\alpha}{\pi(\alpha-1/2)}\Bigr)^{2\alpha}
%    \left(\int_{0}^{\infty}
%    t^{k-2\alpha-1}e^{-t}dt -\int_{0}^{t_1} t^{k-2\alpha-1}e^{-t}dt  \right)\\
  =  \bigl(1+o(1)\bigr)
    \frac{2n\alpha e^{-\alpha C}}{k!}
    \Bigl(\frac{2\alpha}{\pi(\alpha-1/2)}\Bigr)^{2\alpha}
    \left( \Gamma(k-2\alpha)-\int_{0}^{t_1} t^{k-2\alpha-1}e^{-t}dt  \right).
  \end{multline*}
The proof is completed with the observation $t_1 = (1 - o(1))\xi$.
\end{proof}

%Note that by Sterling's approximation for $k=\omega(1)$
%$$k!=\sqrt{2\pi k}\left(\frac{k}{e} \right)^k(1+o(1)) $$
%and
%$$\Gamma(k-2\alpha) = \sqrt{2\pi (k)}\left(\frac{k-2\alpha -1}{e} \right)^{k-2\alpha -1}(1+o(1)) . $$ 
%Furthermore
%$$\Gamma(k-2\alpha)^{-1} \int_{0}^{t_1} t^{k-2\alpha-1}e^{-t}dt =o(1) $$
%and therefore $D_k$ is scale-free distributed with exponent $(-1-2\alpha)$ 
%\begin{equation}
%\label{eq:sequence_large_degree}
%\mathbb{E}[D_k]=n\cdot (1+o(1)) 2\alpha e^{-\alpha C} \left(\frac{2\alpha}{\pi (\alpha -1/2)} \right)^{2\alpha}\cdot k^{-1-2\alpha}. 
%\end{equation}

\begin{proof}[Proof of the first part of Theorem~\ref{thm:degree_seq}]
Set $\beta:= \max\{\frac{3}{5}, 1/(2\alpha), 1-\frac{1}{4\alpha}\}+\varepsilon$, where $\varepsilon > 0$ will be chosen later such that $\beta <1$. The total number of vertices of degree $k$, for any $k$ in the considered range, is at most 
\begin{equation}
\label{eq:degkupper}
D_k(\beta)+|I(\beta)|+e(I(\beta),O(\beta))
\end{equation}
and at least 
\begin{equation}
\label{eq:degklower}
D_k(\beta)-e(I(\beta),O(\beta)) ,
\end{equation}
since every vertex in $I$ could possible have degree $k$, and since each edge counted in $e(I,O)$ may affect the degree of one vertex in $O$. We will argue in the sequel that the contribution of $|I|$ and $e(I,O)$ is with high probability negligible in the above equations. 

First, since $\beta \ge 1 - \frac1{4\alpha}$, by applying Lemma~\ref{lem:few_vertices_I} we obtain that with probability at least $1 - e^{-n^{\Omega(1)}}$,
\begin{equation}
\label{eq:Iupper}
|I(\beta)|=O\left(n^{1-2\alpha(1-\beta)}\right).
\end{equation}
Moreover, Lemma~\ref{lem:few_edges_IO} yields that with high probability
\begin{equation}
\label{eq:eIOupper}
e(I(\beta),O(\beta))=O\left(n^{1-(2\alpha-1)(1-\beta)}\log n\right) .
\end{equation}
It remains to determine the value of $D_k$. Our assumptions on $k$ guarantee that $k\le n^{\delta'}$, where $\delta' < \frac{(2\alpha-1)(1-\beta)}{2\alpha+1}$, provided that $\varepsilon>0$ is sufficiently small. Additionally, we claim that $\delta' < 2(2\beta -1)$. To see this, note that since $a > 1/2$
\[
\frac{(2\alpha-1)(1-\beta)}{2\alpha+1} < 2(2\beta-1) \Leftrightarrow \beta > \frac{6\alpha + 1}{10 \alpha + 3} =: h(\alpha).
\]
But $h(\alpha)$ is increasing, and thus it is maximized at $\alpha = \infty$, where $\lim_{\alpha \to \infty}h(\alpha) = 3/5$. This shows that indeed $\delta' < 2(2\beta -1)$. With all these facts at hand we can apply Lemma~\ref{lem:degree_sequence}, which implies that
\begin{equation}
\label{eq:d_k_ass}
\mathbb{E}[D_k(\beta)]= \Theta(nk^{-(2\alpha+1)}) = \Omega(n^{1- (2\alpha+1)\delta'}) \stackrel{(\delta' < \frac{(2\alpha-1)(1-\beta)}{2\alpha+1})}{=} \omega\left(n^{1-(2\alpha-1)(1-\beta)}\log n\right).
\end{equation}
Note that by Lemma~\ref{lem:lipschitz} the effect on $D_k(\beta)$ if we change the coordinates of of vertex $i$ is at most $c_i:=cn^{1-\beta}+1$ (the plus one is for the vertex itself) as long as $\bar{\mathcal{B}}$ holds. We therefore apply Theorem~\ref{lem:concentration} with $f:=D_k(\beta)$, $t:=n^{3/2-\beta+2\alpha\varepsilon}$ and 'bad' event $\mathcal{B}$ as stated in Lemma~\ref{lem:lipschitz}. Note that $M=n$ and therefore $M\cdot \Pr[\mathcal{B}]=o(1)$. It follows that
\[
	\Pr[|D_k(\beta) - \Exp[D_k(\beta)]| \ge t + o(1)] \le 2e^{-\Omega(t^2/n^{1 + 2(1-\beta)})} + e^{-n^{\Omega(1)}}
	= e^{-n^{\Omega(1)}}.
\]
However, $\beta \ge 1 - \frac1{4\alpha} + \varepsilon$ implies for all $0 \le k \le n^{\delta'}$ that
\[
	\Exp[D_k(\beta)] = \omega \left(n^{1-(2\alpha-1)(1-\beta)}\right) = \omega\left(n^{3/2-\beta+2\alpha\varepsilon}\right) = \omega(t).
\]
This shows that with probability at least $1 - e^{-n^{\Omega(1)}}$ we have for every $k$ in the considered range that $D_k(\beta) = (1 + o(1))\Exp[{D_k(\beta)}]$. Together with~\eqref{eq:degkupper}--\eqref{eq:d_k_ass} the proof of the theorem is completed.
\end{proof}

\subsection{Vertices of Large Degree}
In the previous section we derived the degree sequence for vertices of degree $k\le n^{\delta}$ for some constant $\delta$. For $k>n^{\delta}$ we observe that the radius of almost all vertices of degree $k$ is  concentrated around a specific $r_k$. We effectively show that it suffices to bound the number of vertices of radius at most $r_k$ to get a tight bound on the number of vertices of degree at least $k$. 
%\begin{lemma} Let $C\in \mathbb{R}$ and let $\delta>0$ be an arbitrary small constant. \marginpar{UP: formulierung ok?}Then, with high probability, for all $n^{\delta}\le k \le \frac{n^{1/2\alpha}}{\log n}$, the fraction of vertices of degree at least $k$ in $G_{\alpha, C}(n)$ is
%$$(1+o(1))\left(\frac{2\alpha}{(\alpha-1/2)} \right)^{2\alpha}e^{-\alpha C}k^{-2\alpha} .$$
%\end{lemma}
\begin{proof}[Proof of the second part of Theorem~\ref{thm:degree_seq}]
Let $L_k$ denote the number of vertices of degree at least $k$. Set $r_k:=2(\log (n-1) - \log k+\log(\frac{2\alpha}{\pi(\alpha-1/2)}))$, $q_r:=\mu(B_{r}( R)\cap B_0( R))$ and observe that the expected degree of a vertex with radius $r_k$ is by Lemma~\ref{lem:intersection_area}
$$(n-1)q_{r_k} = (n-1)\frac{2\alpha e^{-r_k/2} }{\pi (\alpha -1/2)}(1\pm O(e^{-(\alpha-1/2)r_k}+e^{-r_k}))=k\left (1\pm O\left(\left(\frac{k}{n}\right)^{2(\alpha -1/2)}+\left(\frac{k}{n}\right)^{2}\right)\right).$$

We set $\varepsilon:= \max\left\{\frac{\log n}{\sqrt{k}},  k^{-(1-2\alpha)^2}\right\}$ and observe that for $\alpha>1/2$ and for all $k\le \frac{n^{\frac{1}{2\alpha}}}{\log n}$
$$O\left(\left(\frac{k}{n}\right)^{2(\alpha -1/2)}+\left(\frac{k}{n}\right)^{2}\right) =o(\varepsilon) .$$

Let $L_k^{<}$ count the vertices of degree at least $k$ and radius at most $r_k-\varepsilon$, $L_k^{>}$ those of degree at least $k$ and radius at least $r_k+\varepsilon$ and $L_k^{\pm}$ those of degree at least $k$ and radius in $[r_k-\varepsilon, r_k+\varepsilon]$. Using those conditions on the radius of the points we can write the expectation of $L_k$ as
\begin{equation}
\label{eq:deg_seq_exp_split}
\mathbb{E}[L_k]=\mathbb{E}[L_k^{<}] +\mathbb{E}[L_k^{>}] +\mathbb{E}[L_k^{\pm}].
\end{equation}
We now show that the first term in \eqref{eq:deg_seq_exp_split} dominates the second and the third. 

Let us first inspect $\mathbb{E}[L_k^{>}] $. By Lemma~\ref{lem:monotone}, the degree distribution of a vertex at distance at least $r_k+\varepsilon$ is dominated by $X\sim Bin(n-1, \mu(B_{r_k+\varepsilon}( R)\cap B_0( R)))$. The expectation of such a r.v. is by Lemma~\ref{lem:intersection_area} and the above observation 
$$\mathbb{E}[X]=(n-1)\frac{2\alpha e^{-\frac{r_k+\varepsilon}{2}}}{\pi(\alpha-1/2)}(1+ o(\varepsilon))= k(1-\varepsilon/2(1+o(1))).$$
By the Chernoff bound the probability that such a vertex has degree at least $k$ is at most 
$$\Pr[X\geq(1+\varepsilon/2(1+o(1)))\mathbb{E}[X]]\le e^{-\frac{\mathbb{E}[X](\varepsilon/2(1+o(1)))^2}{3}}=e^{-\Omega(\log^2n)}=n^{-\Omega(\log n)}$$
and therefore 
\begin{equation}
\label{eq:full_deg_seq_larger_radius}
\mathbb{E}[L_k^>]\le nn^{-\Omega(\log n)}=o(1).
\end{equation}

For points which have distance at most $r_k-\varepsilon$ a similar argument holds as the degree distribution of all those points  dominates  $Y\sim Bin(n-1, \mu(B_{r_k-\varepsilon}( R)\cap B_0( R)))$. This random variable has by Lemma~\ref{lem:intersection_area} expectation 
$$\mathbb{E}[Y]=(n-1)\frac{2\alpha e^{-\frac{r_k-\varepsilon}{2}}}{\pi(\alpha-1/2)}(1+o(\varepsilon))= k(1+\varepsilon/2(1+o(1))).$$
Hence the probability that such a vertex has degree smaller than $k$ is at most
\begin{equation}
\label{eq:full_deg_seq_prob_small_radius}
\Pr[Y\le(1-\varepsilon/2(1+o(1)))\mathbb{E}[Y]]\le e^{-\frac{\mathbb{E}[Y](\varepsilon/2(1+o(1)))^2}{2}}=e^{-\Omega(\log^2n)}=n^{-\Omega(\log n)}
\end{equation}
and therefore $\mathbb{E}[L_k^{<}]=(1-o(1))\mathbb{E}[X_{\le r_k-\varepsilon}]$. Recall that $X_{\le r_k-\varepsilon}$ denotes the number of points with radius at most $r_k-\varepsilon$ and that its expectation is by Lemma~\ref{lem:intersection_area}
\begin{equation}
\label{eq:full_deg_seq_small_radius}
\mathbb{E}[X_{\le r_k-\varepsilon}]=(1+o(1))ne^{-\alpha(R-(r_k-\varepsilon))}=(1+o(1))\left(\frac{2\alpha}{\pi(\alpha-1/2)} \right)^{2\alpha}e^{-\alpha C} n k^{-2\alpha} .
\end{equation}

%% Main work
For the third term in \eqref{eq:deg_seq_exp_split} it suffices to look only at points which have their radius in $[r_k-\varepsilon, r_k+\varepsilon]$ and neglect whether their degree is at least $k$. Clearly
\begin{equation}
\label{eq:deg_seq_in_between}
\mathbb{E}[L_k^{\pm}]\le \mathbb{E}[X_{\le r_k+\varepsilon}]-\mathbb{E}[X_{\le r_k-\varepsilon}]\stackrel{\eqref{eq:full_deg_seq_small_radius}}{=} o(\mathbb{E}[X_{\le r_k-\varepsilon}])   =o(\mathbb{E}[L_k^{<}])
\end{equation}
and we can therefore conclude that 
$$\mathbb{E}[L_k]=(1+o(1))\mathbb{E}[X_{\le r_k-\varepsilon}]=(1+o(1))\left(\frac{2\alpha}{\pi(\alpha-1/2)} \right)^{2\alpha}e^{-\alpha C}nk^{-2\alpha}.$$
Note that for $k\le \frac{n^{\frac{1}{2\alpha}}}{\log n}$, $\mathbb{E}[L_k]=\Omega (\log^{2\alpha}n)$ and therefore by the Chernoff bound
$$\Pr\left[\left|X_{\le r_k-\varepsilon}-\mathbb{E}[X_{\le r_k-\varepsilon}]\right|\geq \frac{1}{\log\log n} \mathbb{E}[X_{\le r_k-\varepsilon}]\right]\le o(n^{-1})$$
which proves the statement of the theorem.  
\end{proof}

\subsection{The Average \& Maximum Degree}
In the subsequent proofs for Theorem~\ref{thm:average_degree} and Theorem~\ref{thm:max_degree}, we present two simple applications of Lemma~\ref{lem:intersection_area}. In particular, we show how it can be used to determine the average degree and the maximum degree in $G_{\alpha, C}(n)$.

\begin{proof}[Proof of Theorem~\ref{thm:average_degree}]
The average degree of a fixed vertex at radius $r$ is by definition 
\begin{equation}
\label{eq:avgkr}
\overline{k}(r)= (n-1) \mu(B_r(R)\cap B_0(R))\stackrel{\text{(Lem.\ \ref{lem:intersection_area})}}{=}(n-1) \left( \frac{2\alpha e^{-r/2}}{\pi(\alpha-1/2)} \Bigl(1\pm O\bigl(e^{-(\alpha-1/2)r}+e^{-r}\bigr)\Bigr) \right).
\end{equation}
The total average degree can be obtained by integrating $\overline{k}(r)p(r)$ over all $0 \le r\le R$. With the facts $2\sinh(x) = e^x(1 - \Theta(e^{-2x}))$ and $2\cosh(x) = e^x(1 + \Theta(e^{-2x}))$, valid for any $x \ge 0$, it follows that
\begin{equation}
\label{eq:prestimate}
	p(r)
	= \frac{ \alpha\sinh(\alpha r)}{\cosh(\alpha R)-1}
	= \alpha e^{\alpha(r-R)} \cdot (1 \pm O(e^{-2\alpha r} + e^{-\alpha R}))
\end{equation}
Note that since $\alpha > 1/2$ and $0 \le r \le R$, the error terms $e^{-2\alpha r} + e^{-\alpha R}$ in~\eqref{eq:prestimate} are dominated by the error terms $e^{-(\alpha-1/2)r}+e^{-r}$ in~\eqref{eq:avgkr}. Hence, by abbreviating $c_\alpha = \frac{2\alpha^2}{\pi(\alpha - 1/2)}$ we obtain
\begin{align*}
\overline{k}
&= \int_0^Rp(r)\overline{k}(r)dr
= (1+o(1)) c_\alpha n \int_0^R  e^{\alpha(r-R)}\cdot e^{-r/2}\left(1 \pm O\left(e^{-(\alpha-1/2)r} + e^{-r}\right)\right) dr.
\end{align*}
The last integral is elementary. In particular, recalling that $R = 2\log n + C$, we obtain that
\[
	\int_0^R  e^{\alpha(r-R)}\cdot e^{-r/2} dr
	= 
	\frac{e^{-\alpha R}}{\alpha - 1/2} \left[e^{(\alpha - 1/2)r}\right]_{r=0}^R
	=
	(1+o(1))\frac{e^{-R/2}}{\alpha - 1/2}
	= (1+o(1))\frac{n^{-1}e^{-C/2}}{\alpha - 1/2}.
\]
The integrals over the error terms are of order
$$O\left(n\int_0^Re^{\alpha(r-R)}e^{-r/2}e^{-(\alpha -1/2)r}dr \right)=O\left(n e^{-\alpha R}\int_0^Rdr \right)=o(1) $$
and
$$O\left(n e^{-\alpha R}\int_0^R e^{r(\alpha -3/2)}dr\right)=O\left(ne^{-\alpha R}e^{R(\alpha -3/2)}\right)=o(1).$$
%Soda only
%The integrals containing the error terms, as they also contain only elementary functions, can be easily shown to be of lower order compared to the expression above.
Finally,
\[
	\overline{k} = (1+o(1)) \frac{c_\alpha \, e^{-C/2}}{\alpha-1/2} = (1+o(1))\frac{2\alpha^2\, e^{-C/2}}{\pi(\alpha - 1/2)^2},
\]
and the theorem follows.
\end{proof}

\begin{proof}[Proof of Theorem~\ref{thm:max_degree}]
Let $X_{\le r}$ denote the number of vertices with radius at most $r$, and define $X_{\geq r}$ similarly. Set $r_0:=(2-1/\alpha)\log n$. We will consider vertices with radius $r\in [r_0-\phi(n), r_0+\phi(n)]$ where $\phi(n)=o(\log \log n)$  tends to infinity as $n\rightarrow \infty$. Note that for sufficiently large $n$
\begin{equation}
\label{eq:max_degree_no_vertices_inside}
\mathbb{E}[X_{\le r_0-\phi(n)}]= n\cdot \mu\left(B_0\left(r_0-\phi(n)\right)\right)\stackrel{(\text{Lem.\ \ref{lem:intersection_area}})}{\le} 2 e^{-\alpha (\phi(n)+C)}=o(1).
\end{equation}
It follows that with high probability there are no vertices with radius at most $r_0-\phi(n)$. Moreover, by applying~\eqref{eq:ball-area}
\begin{equation}
\label{eq:massr0bounds}
\frac{e^{\alpha (\phi(n)-C)}}{2n}\le \mu\left(B_0(r_0+\phi(n))\right)\le \frac{2e^{\alpha (\phi(n)-C)}}{n}.
\end{equation}
Since the number of vertices with radius $\le r_0 + \phi(n)$ is binomially distributed, we infer that
\begin{equation}
\label{eq:max_degree_vertices_inside}
\Pr[X_{\le r_0+\phi(n)}=0]\le \left(1- \mu(B_0(r_0+\phi(n))) \right)^n\le e^{-{e^{\alpha (\phi(n)-C)}}/{2}}=o(1).
\end{equation}
That is, with high probability there is a vertex with radius at most $r_0+\phi(n)$. On the other hand, we will argue that the number of such vertices is very small. To see this, note that~\eqref{eq:massr0bounds} implies that
$$\mathbb{E}[X_{\le r_0+\phi(n)}]\le 2e^{\alpha( \phi(n)-C)},$$
and by applying Markov's inequality we infer, since $\phi(n) = o(\log\log n)$, that
$$\Pr[X_{\le r_0+\phi(n)}>\log n] \le \frac{2e^{\alpha(\phi(n)-C)}}{\log n}=o(1).$$
In other words, the number of vertices with radius larger than $r_0 + \phi(n)$ is at least $n - \log n$ with high probability.

Note that the degree of a vertex at radius $r$ is distributed like $\textrm{Bin}\bigl(n-1, \mu(B_0(R)\cap B_r(R))\bigr)$. As there are with high probability no vertices of radius smaller than $r_0-\phi(n)$, by Lemma~\ref{lem:monotone} the degree distribution of every vertex is with high probability dominated by $\textrm{Bin}(n, p_{r_0}^-)$, where by Lemma~\ref{lem:intersection_area}
$$p_{r_0}^- :=\mu(B_0(R)\cap B_{r_0-\phi(n)}(R))=(1+o(1))\frac{2\alpha}{\pi(\alpha-1/2)}e^{-\frac{r_0-\phi(n)}{2}}.$$
In this case the expected degree of a fixed vertex is at most $np_{r_0}^-\le (1+o(1))\frac{2\alpha}{\pi(\alpha -1/2)}n^{\frac{1}{2\alpha}}e^{\frac{\phi(n)}{2}}$ and by a Chernoff bound and a union bound there is with high probability no vertex of degree larger than $(1+o(1))\frac{4e\alpha}{\pi(\alpha-1/2)}n^{\frac{1}{2\alpha}}e^{\frac{\phi(n)}{2}}$.

The lower bound can be established by the fact (shown in \eqref{eq:max_degree_vertices_inside}) that there is with high probability a vertex of radius at most $r_0+\phi(n)$. There are with high probability at most $\log n$ such vertices which means that we can fix one and have that by Lemma~\ref{lem:monotone} its degree distribution dominates $\textrm{Bin}(n-\log n, p_{r_0}^+)$ where by Lemma~\ref{lem:intersection_area} 
$$p_{r_0}^+:=\mu(B_0(R)\cap B_{r_0+\phi(n)}(R)\setminus B_0(r_0+\phi(n)))=(1+o(1))\frac{2\alpha}{\pi(\alpha-1/2)}e^{-\frac{r_0+\phi(n)}{2}}.$$
Note that the expectation of $\textrm{Bin}(n-\log n, p_0^+)$ is $(1+o(1))np^+_0=(1+o(1))\frac{2\alpha}{\pi(\alpha -1/2)}n^{\frac{1}{2\alpha}}e^{-\frac{\phi(n)}{2}}$ and is therefore by a Chernoff bound with high probability not smaller than $\frac{\alpha}{\pi(\alpha-1/2)}n^{\frac{1}{2\alpha}}e^{-\frac{\phi(n)}{2}}$.
\end{proof}

\footnotesize
\bibliographystyle{plain}
\bibliography{refs}
\end{document}